\documentclass[reqno,centertags,11pt,a4paper]{amsart}
\usepackage{amssymb,mathrsfs}
\usepackage{color,umoline}
\usepackage[dvipsnames]{xcolor}
\usepackage{graphicx}
\usepackage{enumitem}
\usepackage[font=small,labelfont=bf]{caption}
\usepackage{tikz, caption, subcaption}
\usepackage{relsize}
\usepackage[mathscr]{eucal}
\usepackage{verbatim}
\usepackage[draft]{todonotes}
\usepackage{hyperref}
\usepackage[dvipsnames]{xcolor}
\usetikzlibrary{patterns}
\tikzstyle{EDR}=[draw=lightgray,line width=0pt,preaction={clip, postaction={pattern=north east lines, pattern color=gray}}]
\tikzstyle{EDR1}=[draw=lightgray,line width=0pt,preaction={clip, postaction={pattern=north west lines, pattern color=gray}}]
\definecolor{mygray}{gray}{0.95}

\definecolor{mypink1}{rgb}{1.2,1.1,0.9}

\definecolor{mypink2}{rgb}{1.0,0.95 ,0.9}

\definecolor{mypink3}{rgb}{1.0,0.6,0.7}

\numberwithin{equation}{section}

\newtheorem{theorem}{Theorem}[section]
\newtheorem{definition}[theorem]{Definition}
\newtheorem{lemma}[theorem]{Lemma}
\newtheorem{corollary}[theorem]{Corollary}

\newtheorem{proposition}[theorem]{Proposition}
\newtheorem{remark}[theorem]{Remark}

\numberwithin{equation}{section}

\newcommand{\Z}{\mathbb Z}

\newcommand{\beq}{\begin{equation}}
	\newcommand{\eeq}{\end{equation}}
\newcommand{\beqq}{\begin{equation*}}
	\newcommand{\eeqq}{\end{equation*}}
\newcommand{\ben}{\begin{eqnarray}}
	\newcommand{\een}{\end{eqnarray}}
\newcommand{\beno}{\begin{eqnarray*}}
	\newcommand{\eeno}{\end{eqnarray*}}

\begin{document}
\setcounter{page}{1}
\title[Pointwise convergence of noncommutative discrete spherical ergodic averages]
{Pointwise convergence of noncommutative discrete spherical ergodic averages}

\author[]{Cheng Chen}
\address{
Cheng Chen\\
{School of Data Science and Engineering\\
South China Normal University\\
Shanwei 516622 \\
China}}

\email{{chengchen@m.scnu.edu.cn}}

\author[]{Guixiang Hong}
\address{
Guixiang Hong\\
Institute for Advanced Study in Mathematics\\
Harbin Institute of Technology\\
Harbin
150001\\
China}
\email{gxhong@hit.edu.cn}

\author[]{Wenbo Li}
\address{
Wenbo Li\\
Institute for Advanced Study in Mathematics\\
Harbin Institute of Technology\\
Harbin
150001\\
China}
\email{lwbmath@hit.edu.cn}


%


\thanks{}

\subjclass[2010]{Primary  46L51; Secondary 42B20}
\keywords{Noncommutative $L_{p}$ spaces, Discrete spherical maximal inequality, The circle method, Bilateral almost uniform convergence, Strongly ergodic actions}

\date{\today}

\maketitle

\section*{\textbf{Abstract}}
In this paper, we establish spherical maximal inequalities associated with commuting trace-preserving automorphisms on von Neumann algebras, extending Magyar--Stein--Wainger's discrete spherical maximal inequalities to the noncommutative setting. On a noncommutative probability space, we further prove that, under a strong ergodicity assumption, the corresponding discrete spherical averages converge bilaterally almost uniformly to $\tau(x)1$ for every $x\in L_p(\mathcal M)$, $\frac{d}{d-2}<p\leq\infty$, when $d\geq5$.

\section{\textbf{Introduction}}\label{sec1}
Modern discrete harmonic analysis developed out of Bourgain's work on pointwise ergodic theorems along polynomial orbits (cf. \cite{Bou88a,Bou88b,Bou89}).  Somewhat less motivated by ergodic considerations, Stein and his collaborators began to study singular integrals and maximal functions in the discrete context (see, for instance, \cite{SW90,SW99,SW00,MSW02,SW02}).  Their work initiated a rapidly developing subject which combines classical Fourier analysis, ergodic theory, analytic number theory and additive combinatorics; see also \cite{Ion04,IW05,KL19,Hug19,Hug20,KL20a,KL20b,Kra22}.

In the present paper, we are interested in ergodic theory for the discrete spherical averages studied by Magyar, Stein and Wainger \cite{Mag97,MSW02}. Given a finitely supported function $f$ on $\mathbb Z^d$, consider
\[
 M_\lambda f(n)
 =
 \frac{1}{r_d(\lambda^2)}
 \sum_{|m|=\lambda} f(n-m),
 \qquad
 r_d(k)
 =
 \#\{m\in\mathbb Z^d:|m|^2=k\}.
\]
In the whole paper, the parameter $\lambda$ always ranges over the admissible radii, namely those for which
$\lambda^2\in\mathbb N$ and $r_d(\lambda^2)>0$. It is well known that
$r_d(k)$ is irregular for $d\leq4$, whereas
$r_d(k)\approx k^{\frac{d}{2}-1}$ for $d\geq5$ (cf.\ \cite{Wal57,HW60}).
Magyar first obtained a local maximal estimate \cite{Mag97}; later Magyar,
Stein and Wainger \cite{MSW02} proved that, for $d\geq5$ and
$p>\frac{d}{d-2}$,
\[
 \big\|\sup\limits_{\lambda>0}|M_\lambda f|\big\|_{\ell_p(\mathbb Z^d)}
 \leq C_{p,d}\|f\|_{\ell_p(\mathbb Z^d)}.
\]
The ranges of $d$ and $p$ are optimal. Moreover, Ionescu \cite{Ion04} established the restricted weak-type estimate at the endpoint $p=\frac{d}{d-2}$ for $d\geq5$.

Although the maximal inequality transfers to arbitrary measure-preserving $\mathbb Z^d$-actions, a pointwise ergodic theorem also requires convergence on a suitable dense subspace. Actually, Magyar \cite{Mag02} showed that there exist ergodic measure-preserving $\mathbb Z^d$-actions for which the pointwise convergence fails. In the same paper,  he proved convergence under the stronger assumption that the restriction of the action to $q\mathbb Z^d$ is ergodic for every $q\in\mathbb N$, a condition called strong ergodicity. More precisely, let $(\Omega,\mu)$ be a probability space and let $T=(T_1,\ldots,T_d)$ be a commuting family of invertible measure-preserving transformations. Write
\[
 T^m=T_1^{m_1}\cdots T_d^{m_d},
 \qquad
 M_\lambda^T f(\omega)
 =
 \frac{1}{r_d(\lambda^2)}
 \sum_{|m|=\lambda} f(T^m \omega),
\]
If the subaction
\(
 \{T^{qm}:m\in\mathbb Z^d\}
\)
is ergodic for every $q\in\mathbb N$, Magyar's theorem gives, for every
$f\in L_2(\Omega)$,
\[
 M_\lambda^T f(\omega)
 \longrightarrow
 \int_\Omega f\,d\mu
 \qquad\text{for $\mu$-almost every }\omega
\]
as $\lambda\to\infty$ through the admissible radii.

{Noncommutative analysis is a  rapidly developing new field, which is motivated by quantum mechanics, noncommutative geometry, operator algebras. The theory of noncommutative maximal inequalities and pointwise convergences goes back to the weak type $(1,1)$ inequalities of Cuculescu \cite{Cuc71} and Yeadon \cite{Yea77}; see also \cite{Lance, CoN,Ku} for related results on noncommutative ergodic theorems. But its modern period starts with Pisier's vector-valued noncommutative $L_p$-spaces \cite{Pis98} and Junge's noncommutative Doob's maximal inequality \cite{Jun02} (see also \cite{PX97}). Based on these progresses, Junge and Xu \cite{JX07} developed full noncommutative versions of Dunford--Schwartz and Stein's maximal ergodic theorems, while Mei \cite{Mei07} established the operator-valued Hardy--Littlewood maximal inequalities. Further developments include ergodic theorems for free group actions \cite{Hu08,CAD}, groups of polynomial growth \cite{HLW21} and general amenable groups \cite{CW22}. Recently, a full analogue of Bourgain's polynomial ergodic theorem has been obtained in the noncommutative framework \cite{HongLiRicardWang}; see also, \cite{CHW24, HongWang26}. 
Along this research line, inspired by the ergodic theorem for continous spherical averages \cite{Hong13,Hong20}, we investigate in the present paper the noncommutative discrete spherical maximal erogidc inequality and pointwise convergence.

}

Given a von Neumann algebra $\mathcal M$ equipped with a normal semifinite faithful trace $\tau$, one may define the associated noncommutative spaces $L_p(\mathcal M)$ and the dense ideal $\mathcal S_\mathcal M\subset\mathcal M$ with elements of finite trace support; see Section \ref{sec2}. Let $\gamma_1,\ldots,\gamma_d$ be commuting trace-preserving $*$-automorphisms of $\mathcal M$, it is well known that they extend to isometries on $L_p(\mathcal M)$. 
For $n=(n_1,\ldots,n_d)\in\mathbb Z^d$, set $\gamma^n=\gamma_1^{n_1}\cdots\gamma_d^{n_d}$. For $x\in\mathcal S_\mathcal M$, define
\[
 M_\lambda^\gamma x
 =\frac{1}{r_d(\lambda^2)}\sum_{|n|=\lambda}\gamma^n(x).
\]
Our first main result is the following noncommutative ergodic extension of the Magyar--Stein--Wainger maximal  inequality. 

{Given a sequence $(x_m)_m$ in $L_p(\mathcal M)$, the noncommutative maximal $L_p$-norm will be denoted as $\big\|{\sup_{m}}^{+}x_m\big\|_{L_p(\mathcal M)}$ that plays the role of $L_p$-norm of classical maximal function. The precise definition and some properties will be recalled in Section \ref{sec2.2}.}

\begin{theorem}\label{20240929h}
Let $d\geq5$ and $\frac{d}{d-2}<p\leq\infty$.  Then
\[
 \big\|{\sup\limits_{\lambda>0}}^{+}
 M_\lambda^\gamma x\big\|_{L_p(\mathcal M)}
 \leq C_{p,d}\|x\|_{L_p(\mathcal M)},
 \qquad x\in L_p(\mathcal M).
\]
\end{theorem}

After the release of the first version of the present paper, the noncommutative restricted weak-type endpoint estimates for continuous spherical averages at $p=\frac{d}{d-1}$, $d\geq3$, and for discrete spherical averages at $p=\frac{d}{d-2}$, $d\geq5$ were obtained in \cite{HLSW}, extending the classical results of Bourgain \cite{Bou85} and Ionescu \cite{Ion04}, respectively.

Our second main result concerns pointwise convergence of the noncommutative discrete spherical averages on noncommutative probability spaces, that is, von Neumann algebras equipped with a normal faithful tracial state.

Since there is a lack of the notion of {\it point} in the truely noncommutative setting, we will use the notion of {\it bilateral almost uniform} convergence, abbreviated as b.a.u., in place of almost everywhere convergence. This notion and the notion {\it strong ergodicity} used below, will be recalled in Definition \ref{def:bau} and Definition \ref{def:strong} respectively.

\begin{theorem}\label{thm:pointwise-main}
Let $(\mathcal M,\tau)$ be a noncommutative probability space, let $d\geq5$, and suppose that $\gamma$ is strongly ergodic. For every $\frac{d}{d-2}<p\leq\infty$ and every $x\in L_p(\mathcal M)$,
\[
 M_\lambda^\gamma x\longrightarrow\tau(x)1
 \qquad\mathrm{b.a.u.}\quad(\lambda\to\infty).
\]
\end{theorem}

To prove the maximal inequality in Theorem \ref{20240929h},
we follow the strategy of Magyar, Stein and Wainger, decomposing
$M_\lambda=(M_\lambda-N_\lambda)+N_\lambda$ through circle method. For the term $N_\lambda$,
we estimate its maximal $L_p$ norm using the noncommutative sampling
principle and the decay of Gauss sums.
For the error term $M_\lambda-N_\lambda$, we first establish
a rough local $L_p$ estimate and then a sharper local $L_2$ estimate. 
Interpolating between these two estimates yields the desired maximal bound for the error term.
The main difficulty compared
with the commutative argument of Magyar--Stein--Wainger \cite{MSW02} is that
the pointwise supremum appearing in the classical maximal function is no
longer available in the noncommutative setting.  Thus the circle-method
approximation has to be controlled directly in the noncommutative maximal
space $L_p(\mathcal M;\ell_\infty)$.


To prove Theorem \ref{thm:pointwise-main} in the case $\mathcal M=L_\infty(\Omega)$,  in \cite{Mag02} Magyar combines a quantitative maximal estimate with pointwise maximal functions and exceptional sets to obtain almost everywhere convergence by contradiction. In the truly noncommutative setting, we establish the quantitative oscillation estimate in Proposition \ref{prop:oscillation}, but the exceptional-set argument is unavailable, since there is no underlying notion of points. Instead, we will deduce from Proposition \ref{prop:oscillation} that, for $x\in L_2(\mathcal M)$,
\[
 \lim_{R\to\infty}
 \big\|{\sup\limits_{\lambda\geq R}}^{+}
 (M_\lambda^\gamma x-\tau(x)1)\big\|_{L_2(\mathcal M)}=0.
\]
This implies $(M_\lambda^\gamma x-\tau(x)1)_{\lambda}$ belongs to $L_2(\mathcal M;c_0)$ that will be recalled in next section, and thus gives b.a.u. convergence. Finally Theorem \ref{20240929h} and a standard closedness argument extend the conclusion to the stated $L_p$ range.


\begin{remark}\label{rmk:infinite-trace-intro}
{\rm For $d\geq5$ and $\tau(1)=\infty$, ordinary ergodicity suffices to obtain b.a.u.
convergence to $0$ for $x\in L_p(\mathcal M)$ with $\frac{d}{d-2}<p<\infty$;
see Remark \ref{rmk:infinite-trace} for the proof and the corresponding
b.a.u. convergence result.}

\end{remark}

We end this introduction with a brief description of the organization of the paper. Section \ref{sec2} contains the preliminaries on noncommutative $L_p$ spaces, maximal norms, bilateral almost uniform convergence, strongly ergodic actions and the circle partition. Section \ref{sec3} proves a local spherical maximal inequality. In Sections \ref{sec4} and \ref{sec5}, we decompose the spherical averages and estimate the approximation operators. Section \ref{sec:maximal-proof} is devoted to the proof of Theorem \ref{20240929h}. Section \ref{sec:oscillation} establishes the quantitative oscillation estimate by assuming Lemma \ref{lem:finite-denominator}, and completes the proof of Theorem \ref{thm:pointwise-main}. Finally, Section \ref{sec:finite-denominator-proof} is devoted to the proof of Lemma \ref{lem:finite-denominator}. 

Throughout the paper, $C$ stands for a positive constant independent of the main parameters, while it may vary from line to line.\\

\section{\textbf{Preliminaries}}\label{sec2}
\subsection{Noncommutative $L_{p}$ spaces}\label{sec2.1}
In this paper, we focus on the noncommutative $L_{p}$ spaces based on semifinite von Neumann algebras. Let $\mathcal{M}$ be a semifinite von Neumann algebra equipped with a normal semifinite faithful trace $\tau$. Let $\mathcal{S}^{+}_{\mathcal{M}}$ be the set of all $x\in \mathcal{M}_{+}$ such that $\tau(\text{supp}\;x)<\infty$, where $\text{supp}\;x$ denotes the support of $x$. Let $\mathcal{S}_{\mathcal{M}}$ be the linear span of $\mathcal{S}^{+}_{\mathcal{M}}$, then $\mathcal{S}_{\mathcal{M}}$ is weak*-dense $*$-subalgebra of $\mathcal{M}$. For $1\leq p<\infty$, $(\mathcal{S}_{\mathcal{M}},\ \|\cdot\|_{p})$ is a normed space, where
$$\|x\|_{p}:=(\tau(|x|^{p}))^{\frac{1}{p}},\ x\in \mathcal{S}_{\mathcal{M}}$$
with $|x|=(x^{*}x)^{\frac{1}{2}}$ being the modulus of $x$. The noncommutative $L_{p}$ space associated with $(\mathcal{M},\ \tau)$ is the completion of $(\mathcal{S}_{\mathcal{M}},\ \|\cdot\|_{p})$, denoted by $L_{p}(\mathcal{M},\ \tau)$ or simply by $L_{p}(\mathcal{M})$. As usual, we set $L_{\infty}(\mathcal{M}):=\mathcal{M}$ with the operator norm. See \cite{PX03} for more information on noncommutative $L_{p}$ spaces.

Given another semifinite von Neumann algebra $\mathcal{N}$, we can define the von Neumann algebraic tensor product $\mathcal{M}\Bar{\otimes}\mathcal{N}$ as follows: If $\mathcal{M}$ and $\mathcal{N}$ act faithfully on the Hilbert spaces $\mathbb{H}$ and $\mathbb{K}$ respectively, then $\mathcal{M}\Bar{\otimes}\mathcal{N}$ is defined to be the von Neumann algebra in $\mathcal{L}(\mathbb{H}\otimes\mathbb{K})$ generated by all elements of the form $x\otimes y,\ x\in \mathcal{M},\ y\in \mathcal{N}$, where $x\otimes y$ is the bounded linear operator on $\mathbb{H}\otimes\mathbb{K}$ satisfying
$$x\otimes y(\xi\otimes \eta)=x\xi\otimes y\eta,\ \forall \xi\in \mathbb{H},\ \forall \eta\in \mathbb{K}.$$
Although the definition seems to depend on the choice of the Hilbert spaces $\mathbb{H}$ and $\mathbb{K}$, $\mathcal{M}\Bar{\otimes}\mathcal{N}$ is uniquely determined up to $*$-isomorphism.

Throughout the paper, we consider only the special case $\mathcal{N}=L_{\infty}(\Omega)$ with $\Omega$ being a $\sigma$-finite measure space. In this case, the noncommutative space $L_{p}(L_{\infty}(\Omega)\Bar{\otimes}\mathcal{M}),\ {1\leq p< \infty}$ is isometrically isomorphic to the Bochner space $L_{p}(\Omega;L_{p}(\mathcal{M}))$. Here and in the sequel, we will not distinguish these two notions unless {explicitly} stated otherwise.\\

\subsection{The spaces $L_{p}(\mathcal{M};\ell_{\infty})$}\label{sec2.2}
The spaces $L_{p}(\mathcal{M};\ell_{\infty}),\ 1\leq p\leq \infty$ are the noncommutative analogues of the usual Bochner spaces $L_{p}(\Omega;\ell_{\infty})$. They were first introduced by Pisier \cite{Pis98} for hyperfinite von Neumann algebras and then extended to general von Neumann algebras by Junge \cite{Jun02}. The definitions and properties below can be found in \cite{HLW21} and \cite{JX07}.

\begin{definition}[Noncommutative maximal spaces]\label{def:lp-linfty}
For $1\leq p\leq \infty$, the space $L_{p}(\mathcal{M};\ell_{\infty})$ is the space consisting of all sequences $(x_{n})_{n\in \mathbb{N}^{*}}$ in $L_{p}(\mathcal{M})$, each of which admits such a factorization:
$$x_{n}=ay_{n}b,\ a,\ b\in L_{2p}(\mathcal{M}),\ y_{n}\in \mathcal{M}\ \text{with}\ \sup\limits_{n}\|y_{n}\|_{\infty}<\infty.$$
The norm of $(x_{n})_{n\in \mathbb{N}^{*}}$ is then defined as follows:\\
$$\|(x_{n})_{n\in \mathbb{N}^{*}}\|_{L_{p}(\mathcal{M};\ell_{\infty})}:=\text{inf}\big\{\|a\|_{2p}\sup\limits_{n}\|y_{n}\|_{\infty}\|b\|_{2p}\big\},$$
where the infimum runs over all factorizations as above.
\end{definition}
In this setting, the space $L_{p}(\mathcal{M};\ell_{\infty})$ is a Banach space. Moreover, if $(x_{n})_{n\in \mathbb{N}^{*}}$ is a sequence of self-adjoint elements, then $(x_{n})_{n\in \mathbb{N}^{*}}\in L_{p}(\mathcal{M};\ell_{\infty})$ if and only if there exists a positive element $a\in L_{p}(\mathcal{M})$ such that $-a\leq x_{n}\leq a$ for all $n$. In this case,
$$\|(x_{n})_{n\in \mathbb{N}^{*}}\|_{L_{p}(\mathcal{M};\ell_{\infty})}=\text{inf}\{\|a\|_{p},a\in L_{p}(\mathcal{M}),\ a\geq 0\ \text{and}\ \ -a\leq x_{n}\leq a,\ \forall n\in \mathbb{N}^{*}\}.$$
Given the formula above, the norm $\|(x_{n})_{n\in \mathbb{N}^{*}}\|_{L_{p}(\mathcal{M};\ell_{\infty})}$ is commonly denoted by $\big\|{\sup\limits_{n}}^{+}x_{n}\big\|_{p}$ since the latter is more intuitive. It is worth mentioning that the notation ${\sup\limits_{n}}^{+}x_{n}$ does not make any sense in general.

For any index set $I$, we can similarly define the space $L_{p}(\mathcal{M};\ell_{\infty}(I))$ of families $(x_{i})_{i\in I}\subseteq L_{p}(\mathcal{M})$. One can quickly see $(x_{i})_{i\in I}\in L_{p}(\mathcal{M};\ell_{\infty}(I))$ if and only if $(x_{i})_{i\in J}\in L_{p}(\mathcal{M};\ell_{\infty}(J))$ for all finite subsets $J\subseteq I$ and\\
$${\sup\limits_{I\supseteq J\ \text{finite}}}\big\|{\sup\limits_{i\in J}}^{+}x_{i}\big\|_{p}<\infty.$$
In this case, we have
\begin{equation}\label{eq:finite to infinite}
    \big\|{\sup\limits_{i\in I}}^{+}x_{i}\big\|_{p}={\sup\limits_{I\supseteq J\ \text{finite}}}\big\|{\sup\limits_{i\in J}}^{+}x_{i}\big\|_{p}. 
\end{equation}
It is worth mentioning the following uesful lemma 
\begin{lemma}\label{lem:scalar-maximal}
Let $1\leq p\leq\infty$, let $I$ be an arbitrary index set, and let
$(a_i)_{i\in I}$ be a bounded family of complex scalars. For every
$(x_i)_{i\in I}\in L_p(\mathcal M;\ell_\infty(I))$, we have
\[
 \big\|{\sup\limits_{i\in I}}^{+}a_i x_i\big\|_{L_p(\mathcal M)}
 \leq \sup\limits_{i\in I}|a_i|\,
 \big\|{\sup\limits_{i\in I}}^{+}x_i\big\|_{L_p(\mathcal M)}.
\]
\end{lemma}
\begin{proof}
For any factorization $x_i=uy_iv$, we have $a_ix_i=u(a_iy_i)v$ and
\[
 \sup\limits_{i\in I}\|a_iy_i\|_\infty
 \leq \sup\limits_{i\in I}|a_i|\,
       \sup\limits_{i\in I}\|y_i\|_\infty.
\]
Taking the infimum over all factorizations proves the assertion.
\end{proof}
The following lemma records the continuous embedding
$\ell_p(L_p(\mathcal M))\subset L_p(\mathcal M;\ell_\infty)$.
\begin{lemma}\label{lem:L2-maximal-sum}
Let $1\le p<\infty$. For every finite family $(x_j)_{j=1}^m\subset L_p(\mathcal M)$,
we have
$$
\big\|{\sup\limits_{1\leq j\leq m}}^{+}x_j\big\|_p
\leq C \Big(\sum_{j=1}^m\|x_j\|_p^p\Big)^{\frac{1}{p}}.
$$
\end{lemma}
\begin{proof}
    First suppose that $x_j\geq0$ for every $j$, and put
$a=\Big(\sum_{j=1}^m x_j^p\Big)^{\frac{1}{p}}$.
By the operator monotonicity of $t\mapsto t^{\frac{1}{p}}$, we have
$x_j\leq a$ for every $j$. Hence
\[
 \big\|{\sup\limits_{1\leq j\leq m}}^{+}x_j\big\|_p
 \leq\|a\|_p
 =\Big(\sum_{j=1}^m\|x_j\|_p^p\Big)^{\frac{1}{p}}.
\]
For general $x_i$, the conclusion follows by decomposing each $x_i$
into a linear combination of four positive elements and applying
the positive case together with the triangle inequality.
\end{proof}
We have the following fundamental complex interpolation theorem associated with the noncommutative maximal $L_{p}$ norms defined above.
\begin{lemma}
Let $1\leq p_{0}<p_{1}\leq \infty$ and $0<\theta<1$, then we have isometrically
\begin{equation*}
    L_{p}(\mathcal{M};\ell_{\infty}(I))=( L_{p_{0}}(\mathcal{M};\ell_{\infty}(I)),\  L_{p_{1}}(\mathcal{M};\ell_{\infty}(I)))_{\theta},
\end{equation*}
where $\frac{1}{p}=\frac{1-\theta}{p_{0}}+\frac{\theta}{p_{1}}$.\\
\end{lemma}

\subsection{Bilateral almost uniform convergence and the {spaces} $L_p(\mathcal M; c_0)$}\label{sec:almost-uniform}

Almost uniform convergence is the noncommutative analogues of almost everywhere convergence, which was introduced by Lance \cite{Lance}, and its bilateral form
was used by Yeadon \cite{Yea77}; see also \cite[Section 6]{JX07} for more information.
\begin{definition}[Bilateral almost uniform convergence]\label{def:bau}
Let $L_0(\mathcal M)$ denote the space of $\tau$-measurable operators
affiliated with $\mathcal M$. A sequence $(x_n)_{n\geq1}\subset L_0(\mathcal M)$
is said to converge \emph{bilaterally almost uniformly} to $x\in L_0(\mathcal M)$,
abbreviated as $x_n\to x$ b.a.u., if, for every $\varepsilon>0$, there exists
a projection $e\in\mathcal M$ such that
\[
 \tau(1-e)<\varepsilon
 \qquad\text{and}\qquad
 \lim_{n\to\infty}\|e(x_n-x)e\|_\infty=0.
\]
Replacing $\|e(x_n-x)e\|_\infty$ by $\|(x_n-x)e\|_\infty$ defines
\emph{almost uniform convergence}, abbreviated as a.u.
\end{definition}
{In the commutative case}, on a finite measure space, both notions {above coincide} with almost everywhere
convergence {by} virtue of Egorov's theorem.

The following space belongs to the vector-valued framework of
Pisier \cite{Pis98}; its connection with b.a.u. convergence can be found in
\cite{DJ04,JX07}.
\begin{definition}[The space $L_p(\mathcal M;c_0)$]\label{def:lp-c0}
For $1\leq p<\infty$, let $L_p(\mathcal M;c_0)$ be the closure of the
finite sequences in $L_p(\mathcal M;\ell_\infty)$.
\end{definition}
Equivalently, a sequence $(x_n)_{n\geq1}\in L_p(\mathcal M;\ell_\infty)$
belongs to $L_p(\mathcal M;c_0)$ if and only if
\begin{equation}\label{eq:tail Lp}
 \lim_{N\to\infty}
 \big\|{\sup\limits_{n\geq N}}^{+}x_n\big\|_{L_p(\mathcal M)}=0.
\end{equation}

The relevance of this space to individual convergence is the implication
\[
 (x_n)_{n\geq1}\in L_p(\mathcal M;c_0)\quad
 \Longrightarrow\quad
 x_n\longrightarrow0\qquad\mathrm{b.a.u.}
\]

Consequently, if $x\in L_p(\mathcal M)$ and
$(x_n-x)_{n\geq1}\in L_p(\mathcal M;c_0)$, then $x_n\to x$ b.a.u.

\subsection{Strongly ergodic actions}\label{sec:strong-ergodicity}

In this subsection, let $(\mathcal M,\tau)$ be a noncommutative probability space.
Let $\gamma=(\gamma_1,\ldots,\gamma_d),\ d\geq5$ be commuting trace-preserving $*$-automorphisms of $\mathcal{M}$.  For
$m=(m_1,\ldots,m_d)\in\Z^d$, write $\gamma^m=\gamma_1^{m_1}\cdots\gamma_d^{m_d}$.
\begin{definition}[Strong ergodicity]\label{def:strong}
The action $\gamma:\Z^d\curvearrowright\mathcal{M}$ is \emph{strongly ergodic} if,
for every integer $q\geq1$, the subaction $\gamma: q\mathbb{Z}^d\curvearrowright\mathcal{M}$ is ergodic. More precisely:
\begin{equation*}
  \mathcal{M}^{\gamma(q\Z^d)}
  :=\{x\in\mathcal{M}:\gamma^{qm}(x)=x\text{ for every }m\in\Z^d\}
  =\mathbb C1.
\end{equation*}
\end{definition}
The action $\gamma$ induces an unitary Koopman representation
$U=(U^m)_{m\in\Z^d}$ on $L_2(\mathcal M)$, defined by
$U^m x=\gamma^m(x)$. If $e_j$ is the $j$-th coordinate vector, write
$U_j=U^{e_j}$.
We next pass from the fixed-point formulation on \(\mathcal M\) to the corresponding statement for the induced unitary representation on \(L_2(\mathcal M)\).
\begin{lemma}\label{lem:fixed-vectors}
Assume that $\gamma$ is strongly ergodic.  For every integer $q\geq1$,
\begin{equation*}
 L_{2}(\mathcal{M})^{U(q\Z^d)}
 :=\{y\in L_{2}(\mathcal{M}):U^{qm}y=y\text{ for every }m\in\Z^d\}
 =\mathbb C1.
\end{equation*}
\end{lemma}
\begin{proof}
First suppose that $y\in L_2(\mathcal M)^{U(q\Z^d)}$ is self-adjoint, and put
$y_R=y\mathbf{1}_{[-R,R]}(y)\in \mathcal{M}$ for $R>0$.
Functional calculus gives
 \[
 y_R\in\mathcal M^{\gamma(q\Z^d)}=\mathbb C1,
 \qquad
 \|y-y_R\|_2^2
 =\tau(|y|^2\mathbf{1}_{(R,\infty)}(|y|))
 \longrightarrow0.
\] Letting $R\to\infty$ yields $y\in\mathbb C1$.
 
 For general $y$, the conclusion follows by $y=\operatorname{Re}y+i\operatorname{Im}y$
and apply the self-adjoint case to both parts.
\end{proof}
Since the $U_j$ commute, there is a joint projection-valued spectral measure $E$ on
$\mathbb T^d$. By the functional calculus, we have:
\begin{equation*}
  U^m=\int_{\mathbb T^d}e^{2\pi i m\cdot\theta}\,\,\mathrm{d} E(\theta).
\end{equation*}

For $x\in L_{2}(\mathcal{M})$, its scalar spectral measure is
\begin{equation}\label{eq:spectral-measure}
  \nu_x(B)=\langle E(B)x,x\rangle_{L_{2}(\mathcal{M})}
          =\|E(B)x\|_2^2,
  \qquad B\subset\mathbb T^d\text{ Borel}.
\end{equation}

The following lemma describes the spectral measure vanish (except the origin) at rational points when the action is strongly ergodic .
\begin{lemma}\label{lem:rational-atoms}
Assume that $\gamma$ is strongly ergodic, and let $x\in L_2(\mathcal{M})$. Then
\begin{equation*}
 \nu_x(\{r\})=
 \begin{cases}
  |\tau(x)|^2,&r=0,\\
  0,&r\in\mathbb Q^d/\mathbb Z^d\setminus\{0\}.
 \end{cases}
\end{equation*}
\end{lemma}

\begin{proof}
For rational $r$, choose $q\geq1$ with $qr=0$ in $\mathbb T^d$.  By the multiplicativity of the functional calculus, we have:
\[
 U^{qm}E(\{r\})=e^{2\pi iqm\cdot r}E(\{r\})=E(\{r\})
\]
which implies $E(\{r\})L_{2}(\mathcal{M})\subset L_{2}(\mathcal{M})^{U(q\Z^d)}.$
Lemma \ref{lem:fixed-vectors} gives
$E(\{0\})L_{2}(\mathcal{M})=\mathbb C1$ and hence $E(\{0\})x=\tau(x)1$; for $r\ne0$,
\[
 E(\{r\})L_{2}(\mathcal{M})
 \subset\mathbb C1\cap(E(\{0\})L_{2}(\mathcal{M}))^\perp=\{0\}.
\]
The conclusions follow from $\nu_x(\{r\})=\|E(\{r\})x\|_2^2$.
\end{proof}

\subsection{The circle partition}\label{sec2.3}
The circle method is introduced firstly by Hardy and Littlewood in \cite{HL66}, but later it is widely used in discrete harmonic analysis (see \cite{Bou88a, Bou88b, Bou89, Mag97, MSW02} and so on). 

The circle partition is based on the following Farey sequence.
Fix $\Lambda\ge1$. Let $\mathcal F_\Lambda$ denote the Farey
fractions in $(0,1]$ of order $\lfloor\Lambda\rfloor$:
$$
\mathcal F_\Lambda
:=\big\{\frac{a}{q}:
1\le q\le\Lambda,\ 1\le a\le q,\ (a,q)=1\big\}.
$$
in ascending order. For any two successive terms $\frac{a}{q},\ \frac{a'}{q'}$ in $\mathcal{F}_{\Lambda}$, it is well known that
\begin{equation}\label{20241001a}
q+q'>\Lambda\ \text{and}\ \big|\frac{a}{q}-\frac{a'}{q'}\big|=\frac{1}{qq'}.
\end{equation}

For \(q>1\), we associate with each fraction \(\frac{a}{q}\) an interval defined by
$I(\frac{a}{q}):=\{s:\ -\frac{\beta}{q\Lambda}\leq s-\frac{a}{q}< \frac{\alpha}{q\Lambda}\},$
where
\begin{equation}\label{20241001b}
\beta=\beta\big(\frac{a}{q}\big):=\frac{\Lambda}{q+q_{1}}\ \text{and}\ \alpha=\alpha\big(\frac{a}{q}\big):=\frac{\Lambda}{q+q_{2}},  
\end{equation}
with $\frac{a_{1}}{q_{1}},\ \frac{a}{q},\ \frac{a_{2}}{q_{2}}$ being three successive terms in $\mathcal{F}_{\Lambda}\cup\{\frac{0}{1}\}$. For $q=1$, set 
$$
I\big(\frac{1}{1}\big):=\big[1-\frac{\beta}{\Lambda},1+\frac{\alpha}{\Lambda}\big),
$$
with $\alpha=\alpha(\frac{1}{1}):=\frac{\Lambda}{1+[\Lambda]}=:\beta(\frac{1}{1})=\beta$. It follows from \eqref{20241001a} and \eqref{20241001b} that all the coefficients $\alpha, \beta$ defined above lie in $[\frac{1}{2},1)$. And it is easy to check by definition that $\{I(\frac{a}{q})\}_{\frac{a}{q}\in \mathcal{F}_{\Lambda}}$ is a partition of $\mathbb{R}/\mathbb{Z}$. That is,
$$\cup_{\frac{a}{q}\in \mathcal{F}_{\Lambda}}I\big(\frac{a}{q}\big)=(0,1],\quad \mathrm{mod}\ 1,$$
and
$$I\big(\frac{a}{q}\big)\cap I\big(\frac{a'}{q'}\big)=\emptyset,$$
if $\frac{a}{q}\neq \frac{a'}{q'}$. This partition of the  $\mathbb{R}/\mathbb{Z}$ is exactly the key of the circle method.

\section{\textbf{A local spherical maximal inequality}}\label{sec3}
In this section, we introduce the circle method and then prove a local spherical maximal inequality associated with the operator $M_{\lambda}$ on the noncommutative space $L_{p}(\ell_{\infty}(\mathbb{Z}^{d})\Bar{\otimes}\mathcal{M})$.
Recall that 
\begin{equation}
    M_\lambda f(n)=\frac{1}{r_d(\lambda^2)}\sum_{|m|=\lambda}f(n-m).
\end{equation}

The symbol of the operator $M_{\lambda}$ is a function defined on $\mathbb T^d$, expressed as:
\begin{equation*}
\begin{aligned}
m_{\lambda}(\xi)&= \frac{1}{r_{d}(\lambda^{2})}\sum_{|n|=\lambda}e^{2\pi in\xi}=\frac{1}{r_{d}(\lambda^{2})}\sum_{n\in \mathbb{Z}^{d}}e^{2\pi in\xi}\int_{0}^{1}e^{2\pi i(|n|^{2}-\lambda^{2})s}ds\\
&=\frac{e^{2\pi \epsilon \lambda^{2}}}{r_{d}(\lambda^{2})}\int_{0}^{1}e^{-2\pi i\lambda^2s}\sum_{n\in \mathbb{Z}^{d}}e^{2\pi in\xi}e^{-2\pi |n|^{2}(\epsilon-is)}ds,
\end{aligned}
\end{equation*}
where $\epsilon>0$ is a parameter to be determined later, and we have used the fact
\begin{equation*}
    \int_{0}^{1} e^{2\pi i(|m|^{2}-\lambda^{2})s}ds=\Big\{\quad \begin{array}{ll}
        1, & \quad |m|=\lambda \\
        0, & \quad |m|\neq \lambda.
    \end{array} 
\end{equation*}

Here and below, any function defined on $\mathbb{T}^d=\widehat{\mathbb{Z}^d} $ is identified as the $1$-periodic function on $\mathbb{R}^d$. This idea is originally due to Hardy, Littlewood and Ramanujan, changing the combinatorial problem of studying the number $r_d(\lambda^2)$ to an analytic problem through above equation. We recommend {\cite{Wai21}} and the references therein for the circle method.

We have the following decomposition via the circle method for every fixed $\Lambda\geq1$,
$$m_{\lambda}(\xi)=\sum_{\frac{a}{q}\in \mathcal{F}_{\Lambda}}m_{\lambda}^{\frac{a}{q}}(\xi),$$
where
\begin{align}\label{expression of m(a,q)}
m_{\lambda}^{\frac{a}{q}}(\xi):=\frac{e^{2\pi \epsilon \lambda^{2}}}{r_{d}(\lambda^{2})}\int_{I(\frac{a}{q})}e^{-2\pi i\lambda^{2}s}\sum_{n\in \mathbb{Z}^{d}}e^{2\pi in\xi}e^{-2\pi |n|^{2}(\epsilon-is)}ds.
\end{align}

Let $M_\lambda^{\frac{a}{q}}$ be the Fourier multiplier with symbol $m_\lambda^{\frac{a}{q}}$ and $T^{\epsilon}_{s}$ be the convolution operator with kernel $K^{\epsilon}_{s}$ given by $K^{\epsilon}_{s}(m):=e^{-2\pi |m|^{2}(\epsilon -is)}$. Then we have
\begin{align}\label{expression of M(a,q)}
M_{\lambda}^{\frac{a}{q}}f(n)= \frac{e^{2\pi \epsilon \lambda^{2}}}{r_{d}(\lambda^{2})}\int_{I(\frac{a}{q})}e^{-2\pi i\lambda^{2}s}T^{\epsilon}_{s}f(n)ds.
\end{align}

Therefore, the estimate of $M_{\lambda}$ reduces to the estimate of $M_{\lambda}^{\frac{a}{q}}$, which in turn depends on the operator $T^{\epsilon}_{s}$. The key point of the circle method is the observation that the functions $\hat{K}_s^\epsilon$ and $m_\lambda^{\frac{a}{q}}$ behave nicely near rational numbers $\frac{a}{q}$. This in turn follws from the Poisson summation formula and the decay of quadratic Gauss sums. Now we estimate $\hat{K}_s^\epsilon(\xi)$ separately in each interval $I(\frac{a}{q})$ and estimate $m_\lambda^{\frac{a}{q}}$ in next section.

\begin{proposition}
Let $\epsilon=\Lambda^{-2}$. Then for $s\in I(\frac{a}{q})$, we have
\begin{equation}\label{230919a}
\|T^{\epsilon}_{s}f\|_{2}\leq Cq^{-\frac{d}{2}}(\Lambda^{-2}+|t|)^{-\frac{d}{2}}\|f\|_2, \quad \forall f\in L_{2}(\ell_{\infty}(\mathbb{Z}^{d})\Bar{\otimes}\mathcal{M})
\end{equation}
where $t=s-\frac{a}{q}$ and $C$ is a positive constant depending only on $d$.
\end{proposition}

\begin{proof}
Fix \(\frac{a}{q}\in\mathcal F_\Lambda\) and \(s\in I(\frac{a}{q})\). Writing \(s=\frac{a}{q}+t\), we have \(|t|<\frac{1}{q\Lambda}\). The Fourier multiplier associated with \(T_s^\epsilon\) can then be expressed as follows:
$$\hat{K^{\epsilon}_{s}}(\xi)=\sum_{n\in \mathbb{Z}^{d}}e^{-2\pi |n|^{2}(\epsilon-it-i\frac{a}{q})}e^{2\pi in\xi}.$$
Define {a function $h$} on $\mathbb{R}^{d}$ as follows:
$$h(x):=\frac{1}{(2(\epsilon-it))^{\frac{d}{2}}}\sum_{\ell\in \mathbb{Z}^{d}/q\mathbb{Z}^{d}}G\big(\frac{a}{q},\ \ell\big)e^{\frac{-\pi|x-\frac{\ell}{q}|^{2}}{2(\epsilon-it)}},$$
where $G(\frac{a}{q},\ \ell)$ is the quadratic Guass sum
$$G\big(\frac{a}{q},\ \ell\big)=q^{-d}\sum_{n\in \mathbb{Z}^{d}/q\mathbb{Z}^{d}}e^{2\pi i(|n|^{2}\frac{a}{q}+\frac{n\ell}{q})}.$$
An easy application of the Poisson summation formula to the function {$h$} yields
\begin{equation}\label{20240928a}
\hat{K^{\epsilon}_{s}}(\xi)=\frac{1}{(2(\epsilon-it))^{\frac{d}{2}}}\sum_{\ell\in \mathbb{Z}^{d}}G\big(\frac{a}{q},\ \ell\big)e^{\frac{-\pi|\xi-\frac{\ell}{q}|^{2}}{2(\epsilon-it)}}.    
\end{equation}
Noting $|t|<\frac{1}{q\Lambda}$ and $\epsilon=\Lambda^{-2}$, we obtain
\begin{equation}\label{20240928b}
 \frac{\epsilon}{q^{2}(\epsilon^{2}+t^{2})}\geq \frac{1}{q^{2}\Lambda^{2}(\Lambda^{-4}+q^{-2}\Lambda^{-2})}\geq \frac{1}{2},
\end{equation}
and thus
\begin{align*}
&|\hat{K^{\epsilon}_{s}}(\xi)|\leq C(\epsilon+|t|)^{-\frac{d}{2}}q^{-\frac{d}{2}}\sum_{\ell\in \mathbb{Z}^{d}}e^{\frac{-\pi \epsilon|q\xi-\ell|^{2}}{2q^{2}(\epsilon^{2}+t^{2})}}\\
&\qquad\leq C(\epsilon+|t|)^{-\frac{d}{2}}q^{-\frac{d}{2}}\sum_{\ell\in \mathbb{Z}^{d}}e^{-\frac{\pi |q\xi-\ell|^{2}}{4}}\leq C(\epsilon+|t|)^{-\frac{d}{2}}q^{-\frac{d}{2}}.
\end{align*}
Here we have applied the estimate of the normalised Gauss sum $|G(\frac{a}{q},\ell)|=O(q^{-\frac{d}{2}})$. And an application of the Hilbert space-valued Plancherel theorem completes the proof.
\end{proof}

\begin{corollary}\label{20240929i}
Let $d\geq 5$, $\Lambda\geq1$ and $p>\frac{d}{d-2}$ Then we have the following local spherical maximal inequality
$$\big\|{\sup_{\Lambda\leq \lambda\leq 2\Lambda}}^{+}M_{\lambda}f\big\|_{p}\leq C\|f\|_{p},\quad f\in L_{p}(\ell_{\infty}(\mathbb{Z}^{d})\Bar{\otimes}\mathcal{M})$$
for some positive constant $C$ depending on $d$ and $p$.
\end{corollary}

\begin{proof}
Note that
$\|K_s^\epsilon(x)\|_{\ell_{1}(\mathbb{Z}^{d})}\le C\epsilon^{-\frac{d}{2}}$, by noncommutative Young's inequality, we have
$$\|T^{\epsilon}_{s}f\|_{1}\leq \|K^{\epsilon}_{s}\|_{\ell_{1}(\mathbb{Z}^{d})}\|f\|_{1}\leq C\epsilon^{-\frac{d}{2}}\|f\|_{1}.$$
Let $\epsilon=\Lambda^{-2}$, then we obtain, via an easy interpolation with the inequality \eqref{230919a},
$$\|T^{\epsilon}_{s}f\|_{p}\leq Cq^{-\frac{d\theta}{2}}(\Lambda^{-2}+|t|)^{-\frac{d\theta}{2}}\Lambda^{d(1-\theta)}\|f\|_{p},$$
where $t=s-\frac{a}{q}$ with $s\in I(\frac{a}{q})$, $1\leq p\leq 2$ and $0\leq \theta \leq 1$ with $\frac{1}{p}=\frac{\theta}{2}+(1-\theta)$.

When $\frac{d}{d-2}<p\leq2$, or equivalently $1\geq \theta>\frac{4}{d}$, we have
\begin{align}
        &\int_{I(\frac{a}{q})}\|T^{\epsilon}_{s}f\|_{p}ds\leq Cq^{-\frac{d\theta}{2}}\Lambda^{d(1-\theta)}\|f\|_{p}\int_{\bar{I}(\frac{a}{q})}(\Lambda^{-2}+|t|)^{-\frac{d\theta}{2}}dt \nonumber\\ 
        &\qquad\leq Cq^{-\frac{d\theta}{2}}\Lambda^{d(1-\theta)}\Lambda^{d\theta-2}\|f\|_{p}\leq Cq^{-\frac{d\theta}{2}}\Lambda^{d-2}\|f\|_{p},\nonumber
\end{align}
where $\bar{I}(\frac{a}{q})=I(\frac{a}{q})-\frac{a}{q}$. Therefore, by the expression \eqref{expression of M(a,q)} and Lemma \ref{lem:scalar-maximal},
\begin{equation*}
    \begin{aligned}
        &\big\|{\sup\limits_{\Lambda\leq \lambda\leq 2\Lambda}}^{+}M_{\lambda}f\big\|_{p}=\Big\|
        {\sup\limits_{\Lambda\le\lambda\le2\Lambda}}^+\sum_{\frac{a}{q}\in\mathcal{F}_{\Lambda}}M_\lambda^{\frac{a}{q}}f\Big\|_{p} \\
        &\qquad\le\sum_{\frac{a}{q}\in \mathcal{F}_\Lambda}\big\|
        {\sup\limits_{\Lambda\le\lambda\le2\Lambda}}^+M_\lambda^{\frac{a}{q}}f\big\|_{p}\\
        &\qquad\le\sum_{\frac{a}{q}\in\mathcal{F}_{\Lambda}}\int_{I(\frac{a}{q})}{\sup\limits_{\Lambda\leq \lambda\leq 2\Lambda}}\big|\frac{e^{2\pi (\epsilon-is) \lambda^{2}}}{r_{d}(\lambda^{2})}\big|\|T^{\epsilon}_{s}f\|_{p}ds\\
        &\qquad\leq C\Lambda^{2-d}\sum_{1\leq q\leq \Lambda}\sum_{1\leq a\leq q}q^{-\frac{d\theta}{2}}\Lambda^{d-2}\|f\|_{p}\\
        &\qquad\leq C\sum_{1\leq q\leq \Lambda}q^{1-\frac{d\theta}{2}}\|f\|_{p}\leq C\|f\|_{p}.
    \end{aligned}
\end{equation*}

The cases $p>2$ follow by the trivial estimate for $p=\infty$ and interpolation.
\end{proof}

\section{\textbf{Approximation of $M_{\lambda}$ }}\label{sec4}
In this section, we follow Magyar-Stein-Wainger's strategy to approximate the operator $M_\lambda$ by another discrete operator $N_\lambda$ and then deduce the sharp estimate of the difference $M_{\lambda}-N_{\lambda}$. The  latter operator is closer to the continuous spherical averages and can be handled via the sampling principle.

Substituting \eqref{20240928a} into
\eqref{expression of m(a,q)}, we obtain
\[
m_\lambda^{\frac{a}{q}}(\xi)
=e^{-2\pi i\lambda^2\frac{a}{q}}
\sum_{\ell\in\mathbb Z^d}
G\big(\frac{a}{q},\ell\big)J_\lambda^{\frac{a}{q}}\big(\xi-\frac{\ell}{q}\big),
\]
with $$J_\lambda^{\frac{a}{q}}(\xi)
=\frac{e^{2\pi\epsilon\lambda^2}}{r_d(\lambda^2)}
\int_{\bar I(\frac{a}{q})}
\frac{e^{-2\pi i\lambda^2t}}{(2(\epsilon-it))^{\frac{d}{2}}}
e^{-\frac{\pi|\xi|^2}{2(\epsilon-it)}}\,dt,$$
where $\bar I(\frac{a}{q})=I(\frac{a}{q})-\frac{a}{q}$. Set \(Q=(-\frac{1}{2},\frac{1}{2}]^d\). Let $\phi\in C_c^\infty(\mathbb R^d)$ be supported in $Q/2$ and satisfy $\phi\equiv1$ on $Q/4$. Define
\begin{align}
l_{\lambda}^{\frac{a}{q}}(\xi)
&= e^{-2\pi i\lambda^{2}\frac{a}{q}}
\sum_{\ell\in\mathbb{Z}^{d}}
G\big(\frac{a}{q},\ell\big)
\phi_{q}\big(\xi-\frac{\ell}{q}\big)
J^{\frac{a}{q}}_{\lambda}\big(\xi-\frac{\ell}{q}\big) \nonumber,\\
n_{\lambda}^{\frac{a}{q}}(\xi)
&= e^{-2\pi i\lambda^{2}\frac{a}{q}}
\sum_{\ell\in\mathbb{Z}^{d}}
G\big(\frac{a}{q},\ell\big)
\phi_{q}\big(\xi-\frac{\ell}{q}\big)
J_{\lambda}\big(\xi-\frac{\ell}{q}\big),\label{eq_rep_n_lambda}
\end{align}
where $\phi_{q}(\xi)=\phi(q\xi)$ and $$J_{\lambda}(\xi):=\frac{e^{2\pi \epsilon \lambda^{2}}}{r_{d}(\lambda^{2})}\int^{\infty}_{-\infty}\frac{e^{-2\pi i\lambda^{2}t}}{(2(\epsilon-it))^{\frac{d}{2}}}e^{\frac{-\pi|\xi|^{2}}{2(\epsilon-it)}}dt.$$

Let $L_{\lambda}^{\frac{a}{q}}$ and $N_{\lambda}^{\frac{a}{q}}$ be the convolution operators acting on $\mathcal{S}_\mathcal{M}$-valued finitely supported functions on $\mathbb{Z}^{d}$, whose Fourier multipliers are $l_{\lambda}^{\frac{a}{q}}$ and $n_{\lambda}^{\frac{a}{q}}$, respectively. The following proposition quantifies the approximation of $M_\lambda^{\frac{a}{q}}$ by $N_\lambda^{\frac{a}{q}}$.
\begin{proposition}\label{20240929j}
    Let $d\ge5$ and $\Lambda\ge1$, and set $\epsilon=\Lambda^{-2}$. Then there exists a constant $C$ depending only on $d$ and $\phi$ such that for any $f\in L_2(\ell_\infty(\mathbb Z^d)\bar\otimes\mathcal M)$,

{\rm (i)}
\begin{align*}
\sum_{\frac{a}{q}\in\mathcal F_\Lambda}
\big\|{\sup_{\Lambda\le\lambda\le2\Lambda}}^{+}
(M_\lambda^{\frac{a}{q}}-L_\lambda^{\frac{a}{q}})f\big\|_2
\le C\Lambda^{2-\frac{d}{2}}\|f\|_2,
\end{align*}

{\rm (ii)}
\begin{align*}
\sum_{\frac{a}{q}\in\mathcal F_\Lambda}
\big\|{\sup_{\Lambda\le\lambda\le2\Lambda}}^{+}
(N_\lambda^{\frac{a}{q}}-L_\lambda^{\frac{a}{q}})f\big\|_2
\le C\Lambda^{2-\frac{d}{2}}\|f\|_2.
\end{align*}
\end{proposition}

\begin{proof}
(i). Let $R^{\frac{a}{q}}_{t}$ be the convolution operator acting on $\mathcal{S}_\mathcal{M}$-valued finitely supported functions on $\mathbb{Z}^{d}$ with Fourier multiplier
$$r^{\frac{a}{q}}_{t}(\xi)=\sum_{\ell\in \mathbb{Z}^{d}}G\big(\frac{a}{q},\ \ell\big)\big(1-\phi_{q}\big(\xi-\frac{\ell}{q}\big)\big)e^{\frac{-\pi|\xi-\frac{\ell}{q}|^{2}}{2(\epsilon-it)}}.$$
Since each term in the sum is supported in $\{\xi\in\mathbb T^d:\;  |\xi-\frac{\ell}{q}|\geq \frac{c}{q}\}$ for some $c>0$, we have
\begin{equation*}
    \begin{aligned}
       |r^{\frac{a}{q}}_{t}(\xi)|&\leq Cq^{-\frac{d}{2}}\sum_{\ell\in \mathbb{Z}^{d},\ |\ell-q\xi|\geq c}e^{\frac{-\pi\epsilon|q\xi-\ell|^{2}}{2q^{2}(\epsilon^{2}+t^{2})}}\leq Cq^{-\frac{d}{2}}\big(\frac{\epsilon}{q^{2}(\epsilon^{2}+t^{2})}\big)^{-\frac{d}{2}}\\
       &\leq Cq^{-\frac{d}{2}}\big(\frac{\epsilon}{q^{2}(\epsilon^{2}+t^{2})}\big)^{-\frac{d}{4}}\leq C\big(\frac{\epsilon}{\epsilon^{2}+t^{2}}\big)^{-\frac{d}{4}},
    \end{aligned}
\end{equation*}
where the third inequality follows from the estimate \eqref{20240928b} when $\epsilon=\Lambda^{-2}$.
Therefore,
$\|R^{\frac{a}{q}}_{t}f\|_2\leq C(\frac{\epsilon}{\epsilon^{2}+t^{2}})^{-\frac{d}{4}}\|f\|_2$. 
Substituting the definition of $J_\lambda^{\frac{a}{q}}$ into
$m_\lambda^{\frac{a}{q}}-l_\lambda^{\frac{a}{q}}$ and applying the Fourier inversion formula, we obtain
\begin{equation*}
   (M_{\lambda}^{\frac{a}{q}}-L_{\lambda}^{\frac{a}{q}})f=\frac{e^{2\pi \epsilon \lambda^{2}}e^{-2\pi i\lambda^{2}\frac{a}{q}}}{r_{d}(\lambda^{2})}\int_{\Bar{I}(\frac{a}{q})}\frac{e^{-2\pi i\lambda^{2}t}}{(2(\epsilon-it))^{\frac{d}{2}}}R^{\frac{a}{q}}_{t}fdt,
\end{equation*}
Using Lemma \ref{lem:scalar-maximal} and the preceding estimate
for $R_t^{\frac{a}{q}}$, we obtain:
\begin{equation*}
    \begin{aligned}
        &\big\|{\sup_{\Lambda\leq \lambda\leq 2\Lambda}}^{+}(M_{\lambda}^{\frac{a}{q}}-L_{\lambda}^{\frac{a}{q}})f\big\|_2\leq {\sup\limits_{\Lambda\leq \lambda\leq 2\Lambda}}\big|\frac{e^{2\pi \epsilon \lambda^{2}}e^{-2\pi i\lambda^{2}\frac{a}{q}}}{r_{d}(\lambda^{2})}\big|\cdot \Big\|{\sup\limits_{\Lambda\leq \lambda\leq 2\Lambda}}^{+}\int_{\Bar{I}(\frac{a}{q})}\frac{e^{-2\pi i\lambda^{2}t}}{(2(\epsilon-it))^{\frac{d}{2}}}R^{\frac{a}{q}}_{t}fdt\Big\|_2\\
        &\qquad\leq C\Lambda^{2-d}\int_{\Bar{I}(\frac{a}{q})}{\sup\limits_{\Lambda\leq \lambda\leq 2\Lambda}}\big|\frac{e^{-2\pi i\lambda^{2}t}}{(2(\epsilon-it))^{\frac{d}{2}}}\big|\cdot \|R^{\frac{a}{q}}_{t}f\|_2dt\\
        &\qquad\leq C\Lambda^{2-d}\int_{\Bar{I}(\frac{a}{q})}(\epsilon^{2}+t^{2})^{-\frac{d}{4}}\big(\frac{\epsilon}{\epsilon^{2}+t^{2}}\big)^{-\frac{d}{4}}dt\|f\|_2=C\Lambda^{2-\frac{d}{2}}\big|\Bar{I}\big(\frac{a}{q}\big)\big|\cdot\|f\|_2
    \end{aligned}
\end{equation*}
Summing over $\frac{a}{q}\in\mathcal F_\Lambda$ and using
$\sum_{\frac{a}{q}\in\mathcal F_\Lambda}|\bar I(\frac{a}{q})|=1$
proves (i).

(ii). Let $S^{\frac{a}{q}}_{t}$ be the convolution operator acting on $\mathcal{S}_\mathcal{M}$-valued finitely supported functions on $\mathbb{Z}^{d}$ with Fourier multiplier
$$s^{\frac{a}{q}}_{t}(\xi)=\sum_{\ell\in \mathbb{Z}^{d}}G\big(\frac{a}{q},\ \ell\big)\phi_{q}\big(\xi-\frac{\ell}{q}\big)e^{\frac{-\pi|\xi-\frac{\ell}{q}|^{2}}{2(\epsilon-it)}}.$$
For each $\xi\in\mathbb R^d$, at most one term in the above
sum is nonzero. Thus $|s_t^{\frac{a}{q}}(\xi)|\le Cq^{-\frac{d}{2}}$,
and Plancherel's theorem gives
$\|S_t^{\frac{a}{q}}f\|_2\le Cq^{-\frac{d}{2}}\|f\|_2$. 
Substituting the definitions of $J_\lambda^{\frac{a}{q}}$ and $J_\lambda$ into $n_\lambda^{\frac{a}{q}}-l_\lambda^{\frac{a}{q}}$ and applying Fourier inversion, we obtain
\begin{equation*}
   (N_{\lambda}^{\frac{a}{q}}-L_{\lambda}^{\frac{a}{q}})f=\frac{e^{2\pi \epsilon \lambda^{2}}e^{-2\pi i\lambda^{2}\frac{a}{q}}}{r_{d}(\lambda^{2})}\int_{\mathbb{R}\setminus \Bar{I}(\frac{a}{q})}\frac{e^{-2\pi i\lambda^{2}t}}{(2(\epsilon-it))^{\frac{d}{2}}}S^{\frac{a}{q}}_{t}fdt.
\end{equation*}
Using Lemma \ref{lem:scalar-maximal} and the preceding estimate for $S_t^\frac{a}{q}$, we obtain:
\begin{equation*}
    \begin{aligned}
        &\big\|{\sup\limits_{\Lambda\leq \lambda\leq 2\Lambda}}^{+}(N_{\lambda}^{\frac{a}{q}}-L_{\lambda}^{\frac{a}{q}})f\big\|_2\\
        &\qquad\leq {\sup\limits_{\Lambda\leq \lambda\leq 2\Lambda}}\big|\frac{e^{2\pi \epsilon \lambda^{2}}e^{-2\pi i\lambda^{2}\frac{a}{q}}}{r_{d}(\lambda^{2})}\big|\cdot\Big\|{\sup\limits_{\Lambda\leq \lambda\leq 2\Lambda}}^{+}\int_{\mathbb{R}\setminus \Bar{I}(\frac{a}{q})}\frac{e^{-2\pi i\lambda^{2}t}}{(2(\epsilon-it))^{\frac{d}{2}}}S^{\frac{a}{q}}_{t}fdt\Big\|_2\\
        &\qquad\leq C\Lambda^{2-d}\int_{\mathbb{R}\setminus \Bar{I}(\frac{a}{q})}{\sup\limits_{\Lambda\leq \lambda\leq 2\Lambda}}\big|\frac{e^{-2\pi i\lambda^{2}t}}{(2(\epsilon-it))^{\frac{d}{2}}}\big|\cdot\|S^{\frac{a}{q}}_{t}f\|_2dt\\
        &\qquad\leq C\Lambda^{2-d}\int_{\mathbb{R}\setminus \Bar{I}(\frac{a}{q})}(\epsilon^{2}+t^{2})^{-\frac{d}{4}}q^{-\frac{d}{2}}dt\|f\|_2\\
        &\qquad\leq C\Lambda^{2-d}\int_{|t|\geq \frac{c}{q\Lambda}}|t|^{-\frac{d}{2}}q^{-\frac{d}{2}}dt\|f\|_2\leq C\Lambda^{1-\frac{d}{2}}q^{-1}\|f\|_2.
    \end{aligned}
\end{equation*}
Therefore,
\begin{equation*}
\begin{aligned}
    &\sum_{\frac{a}{q}\in \mathcal{F}_{\Lambda}}\big\|{\sup\limits_{\Lambda\leq \lambda\leq 2\Lambda}}^{+}(N_{\lambda}^{\frac{a}{q}}-L_{\lambda}^{\frac{a}{q}})f\big\|_2\leq C\Lambda^{1-\frac{d}{2}}\|f\|_2\sum_{\frac{a}{q}\in \mathcal{F}_{\Lambda}}q^{-1}\leq C\Lambda^{2-\frac{d}{2}}\|f\|_2.
\end{aligned}
\end{equation*}
\end{proof}

\section{\textbf{Maximal estimates for $N_{\lambda}$}}\label{sec5}

In this section, we establish a maximal inequality for
$N_\lambda$ using a noncommutative sampling principle.
Let
\begin{equation}\label{def_F}
    \mathcal F
:=\big\{\frac{a}{q}:\ a,q\in\mathbb N,\ 1\le a\le q,\ (a,q)=1\big\},
\end{equation}
and define
$N_\lambda:=\sum_{\frac{a}{q}\in\mathcal F}N_\lambda^{\frac{a}{q}}$.
The convergence of this series will be justified below. Our main result is the following.

\begin{proposition}\label{20240929c}
If $d\geq 5$ and $\frac{d}{d-2}<p\leq 2$, then we have
\begin{equation}\label{20240929l}
    \big\|{\sup_{0<\lambda<\infty}}^+N_\lambda f\big\|_p
\le C\|f\|_p, \quad\forall f\in L_p(\ell_\infty(\mathbb Z^d)\bar\otimes\mathcal M).
\end{equation}

\end{proposition}

Before giving the proof, we need to do some preparation.
Let $\Phi$ be a smooth function on $\mathbb R^d$ supported in $Q/q$, and define
\begin{equation}\label{20240929a}
  \Phi_{\text{per}}(\xi):=\sum_{\ell\in \mathbb{Z}^{d}}\Phi\big(\xi-\frac{\ell}{q}\big).
\end{equation}

Let $T_{\Phi_\mathrm{per}}$ be the convolution operator acting on functions on $\mathbb{Z}^{d}$ with Fourier multiplier $\Phi_{\text{per}}$. That is,
$$\sum_{n\in \mathbb{Z}^{d}}T_{\Phi_\mathrm{per}}(f)(n)e^{-2\pi in\xi}=\Phi_{\text{per}}(\xi)\sum_{n\in \mathbb{Z}^{d}}f(n)e^{-2\pi in\xi}$$
for any finitely supported function $f$ on $\mathbb{Z}^{d}$.

As above, one may regard a smooth symbol on $\mathbb R^d$ supported in $Q/q$ as a symbol on $\mathbb T^d$, and then has the following noncommutative sampling principle for noncommutative maximal norms. 

\begin{lemma}\label{20240929d}
Let $\{\Phi_{i}\}_{i\in I}$ be a collection of smooth functions on $\mathbb{R}^{d}$ supported in the cube $Q$ and let $\mathcal{T}_{\Phi_{i}}$ be the convolution operator acting on functions on $\mathbb{R}^{d}$ with Fourier multiplier $\Phi_{i}$. Fix $1\le p\le\infty$. Suppose that the estimate
$$\big\|{\sup\limits_{i\in I}}^{+}\mathcal{T}_{\Phi_{i}}(g)\big\|_p\leq C\|g\|_p \quad \forall g\in L_{p}(L_{\infty}(\mathbb{R}^{d})\Bar{\otimes}\mathcal{M})$$
holds for any semifinite von Neumann algebra $\mathcal{M}$. Then the following estimate
$$\big\|{\sup\limits_{i\in I}}^{+}T_{(\Phi_{i,q})_\mathrm{per}}(f)\big\|_p\leq C\|f\|_p \quad \forall f\in L_{p}(\ell_{\infty}(\mathbb{Z}^{d})\Bar{\otimes}\mathcal{M}) $$
holds for any semifinite von Neumann algebra $\mathcal{M}$, where $\Phi_{i,q}(\xi):=\Phi_{i}(q\xi)$.
\end{lemma}


This lemma is the high-dimensional version of \cite[Lemma 3.5]{CHW24}. Since the proof is quite similar, so we omit it here.

Recall $\Phi\in C^\infty(\mathbb R^d)$ is
supported in $Q/q$. We next generalize \eqref{20240929a}
by considering multipliers of the form
\begin{equation}\label{20240929b}
 m(\xi):=\sum_{\ell\in \mathbb{Z}^{d}}a_{\ell}\Phi\big(\xi-\frac{\ell}{q}\big)
\end{equation}
with (i) $\sum_{n\in \mathbb{Z}^{d}}|\hat{\Phi}(n)|\leq A$ for some positive constant $A$, and (ii) $a_{\ell}$ is $q\mathbb{Z}^{d}$ periodic, i.e., $a_{\ell}=a_{\ell'}$ if $\ell\equiv \ell'\ \text{mod}\ q\mathbb{Z}^{d}$.

\begin{lemma}\label{20240929e}
Let $T_m$ denote the Fourier multiplier
operator on $\mathbb Z^d$ with symbol $m$ given by
\eqref{20240929b}. Then, for $1\le p\le2$, we have
$$\|T_mf\|_{p}\leq A\|(a_{\ell})_{\ell}\|_{\ell_{\infty}(\mathbb{Z}^{d})}^{2-\frac{2}{p}}\|(\hat{a}_{k})_{k}\|_{\ell_{\infty}(\mathbb{Z}^{d})}^{\frac{2}{p}-1}\|f\|_p,\quad\forall f\in L_{p}(\ell_{\infty}(\mathbb{Z}^{d})\Bar{\otimes}\mathcal{M})$$
where $\hat{a}_{k}=\sum_{\ell\in \mathbb{Z}^{d}/q\mathbb{Z}^{d}}a_{\ell}e^{2\pi ik\frac{\ell}{q}}$.
\end{lemma}

\begin{proof}
For every $\xi\in \mathbb{R}^{d}$, at most one term in the sum \eqref{20240929b} is nonzero since $\Phi$ is supported in $\frac{Q}{q}$. As a result,
$$\|m\|_{L_{\infty}(Q)}\leq \|(a_{\ell})_{\ell}\|_{\ell_{\infty}(\mathbb{Z}^{d})}\|\Phi\|_{L_{\infty}(Q)}\leq \|(a_{\ell})_{\ell}\|_{\ell_{\infty}(\mathbb{Z}^{d})}\sum_{n\in \mathbb{Z}^{d}}|\hat{\Phi}(n)|\leq A\|(a_{\ell})_{\ell}\|_{\ell_{\infty}(\mathbb{Z}^{d})}.$$
By the Hilbert space-valued Plancherel theorem, we obtain for $ f\in L_{2}(\ell_{\infty}(\mathbb{Z}^{d})\Bar{\otimes}\mathcal{M})$:
$$\|T_mf\|_{2}\leq \|m\|_{L_{\infty}(Q)}\|f\|_2\leq A\|(a_{\ell})_{\ell}\|_{\ell_{\infty}(\mathbb{Z}^{d})}\|f\|_2.$$
On the other hand, the convolution kernel $K$ of $T_m$ is given by
\begin{equation*}
\begin{aligned}
   K(n)&=\int_{Q}\sum_{\ell\in \mathbb{Z}^{d}}a_{\ell}\Phi\big(\xi-\frac{\ell}{q}\big)e^{2\pi in\xi}d\xi\\
    &=\sum_{\ell\in \mathbb{Z}^{d}}a_{\ell}e^{2\pi in\frac{\ell}{q}}\int_{Q-\frac{\ell}{q}}\Phi(\xi)e^{2\pi in\xi}d\xi\\
    &=\sum_{r\in \mathbb{Z}^{d}/q\mathbb{Z}^{d}}a_{r}e^{2\pi in\frac{r}{q}}\sum_{m\in \mathbb{Z}^d}\int_{Q-\frac{r}{q}-m}\Phi(\xi)e^{2\pi in\xi}d\xi\\
    &=\sum_{r\in \mathbb{Z}^{d}/q\mathbb{Z}^{d}}a_{r}e^{2\pi in\frac{r}{q}}\hat{\Phi}(-n)=\hat{a}_{n}\hat{\Phi}(-n).
\end{aligned}
\end{equation*}
The third equality above holds since $a_{\ell}$ is $q\mathbb{Z}^{d}$ periodic and $\Phi$ is supported in $Q/q$.
By the noncommutative Young's inequality, we obtain, for every $f\in L_1(\ell_\infty(\mathbb Z^d)\bar\otimes\mathcal M)$,
$$\|T_mf\|_1\leq \|K\|_{\ell_{1}(\mathbb{Z}^{d})}\|f\|_1\leq A\|(\hat{a}_{n})_{n}\|_{\ell_{\infty}(\mathbb{Z}^{d})}\|f\|_1.$$
Interpolation between the $L_1$ and $L_2$ estimates completes the proof.
\end{proof}
Let $\sigma_\lambda$ denote the normalized surface measure on
$\{x\in\mathbb R^d:|x|=\lambda\}$, with Fourier transform
$\hat\sigma_\lambda$.
By \cite[Lemma 6.1]{MSW02}, $J_\lambda(\xi)
=\frac{c_d\lambda^{d-2}}{r_d(\lambda^2)}
 \hat\sigma_\lambda(\xi)$ where $c_d:=\frac{\pi^{\frac{d}{2}}}{\Gamma(\frac{d}{2})}.$
Substituting this identity into \eqref{eq_rep_n_lambda}, we obtain
$$n_{\lambda}^{\frac{a}{q}}(\xi)=\frac{c_{d}\lambda^{d-2}}{r_{d}(\lambda^{2})}e^{-2\pi i\lambda^{2}\frac{a}{q}}\sum_{\ell\in \mathbb{Z}^{d}}G\big(\frac{a}{q},\ \ell\big)\phi_{q}\big(\xi-\frac{\ell}{q}\big)\hat{\sigma}_{\lambda}\big(\xi-\frac{\ell}{q}\big).$$

 Now we are ready to prove Proposition \ref{20240929c}.

\begin{proof}
Recall $\phi\in C_c^\infty(\mathbb R^d)$ is supported in $Q/2$.
Let $\psi$ be a smooth function on $\mathbb{R}^{d}$ supported in $Q$ with $\psi=1$ on $\frac{Q}{2}$. Then $\phi\psi=\phi$. 
Define two Fourier multipliers $T_{g^{\frac{a}{q}}},\ T_{(\psi_q\hat{\sigma}_\lambda)_\mathrm{per}}$ on functions on $\mathbb Z^d$ with symbol respectively,
\begin{flalign*}
   & g^{\frac{a}{q}}(\xi)=\sum_{\ell\in \mathbb{Z}^{d}}G\big(\frac{a}{q},\ \ell\big)\phi_{q}\big(\xi-\frac{\ell}{q}\big)\\
   & (\psi_q\hat{\sigma}_\lambda)_{\mathrm{per}}(\xi)
 =\sum_{\ell\in \mathbb{Z}^{d}}\psi_{q}\big(\xi-\frac{\ell}{q}\big)\hat{\sigma}_{\lambda}\big(\xi-\frac{\ell}{q}\big),
\end{flalign*}
Then
$$ N^{\frac{a}{q}}_{\lambda}f=\frac{c_{d}\lambda^{d-2}}{r_{d}(\lambda^{2})}e^{-2\pi i\lambda^{2}\frac{a}{q}}T_{(\psi_q\hat{\sigma}_\lambda)_\mathrm{per}}T_{g^\frac{a}{q}}f.$$
Let $\mathcal T_{\psi\hat\sigma_{q^{-1}\lambda}}$ be the Fourier multiplier
operator on $\mathbb R^d$ with symbol
$\psi\hat\sigma_{q^{-1}\lambda}$. Combining the noncommutative spherical maximal inequality (see \cite[Proposition 4.1]{Hong13}) with the noncommutative Young's inequality, we obtain the following estimate for every $g\in L_p(L_\infty(\mathbb{R}^d)\bar\otimes\mathcal{M})$:
$$\big\|{\sup\limits_{0< \lambda<\infty}}^{+}\mathcal{T}_{\psi\hat{\sigma}_{q^{-1}\lambda}}g\big\|_p\leq C\|g\|_p,\quad p>\frac{d}{d-1},\ d\geq 3.$$
Consequently, by Lemma \ref{20240929d}, we obtain, for every $f\in L_p(\ell_\infty(\mathbb{Z}^d)\bar\otimes\mathcal{M})$:
\begin{equation}\label{20240929f}
 \big\|{\sup\limits_{0< \lambda<\infty}}^{+}T_{(\psi_q\hat{\sigma}_\lambda)_\mathrm{per}}f\big\|_p\leq C\|f\|_p,\quad p>\frac{d}{d-1},\ d\geq 3.
\end{equation}
On the other hand, $\phi_{q}$ is a smooth function supported in $\frac{Q}{q}$ and
$$\sum_{n\in \mathbb{Z}^{d}}|\hat{\phi}_{q}(n)|=\sum_{n\in \mathbb{Z}^{d}}q^{-d}|\hat{\phi}(q^{-1}n)|\leq Cq^{-d}\sum_{n\in \mathbb{Z}^{d}}(1+|q^{-1}n|)^{-d-1}\leq C,$$
where $C$ is a positive constant independent of $q$.
Noting the Gauss sum $G(\frac{a}{q},\ \ell)$ is $q\mathbb{Z}^{d}$ periodic, we obtain, by Lemma \ref{20240929e},
\begin{equation}
    \|T_{g^{\frac{a}{q}}}f\|_p\leq C\sup\limits_{\ell}\big|G\big(\frac{a}{q},\ \ell\big)\big|^{2-\frac{2}{p}}\sup\limits_{k}\big|\hat{G}\big(\frac{a}{q},\ k\big)\big|^{\frac{2}{p}-1}\|f\|_p
\end{equation}
for all $1\leq p\leq 2$, where
\begin{equation*}
    \begin{aligned}
     &\hat{G}\big(\frac{a}{q},\ k\big):=\sum_{\ell\in \mathbb{Z}^{d}/q\mathbb{Z}^{d}}e^{2\pi ik\frac{\ell}{q}}G\big(\frac{a}{q},\ \ell\big)\\
&\qquad=\frac{1}{q^{d}}\sum_{n\in\mathbb{Z}^{d}/q\mathbb{Z}^{d}}\sum_{\ell\in\mathbb{Z}^{d}/q\mathbb{Z}^{d}}e^{2\pi ik\frac{\ell}{q}}e^{2\pi i|n|^{2}\frac{a}{q}}e^{2\pi in\frac{\ell}{q}}=e^{2\pi i|k|^{2}\frac{a}{q}}.
    \end{aligned}
\end{equation*}
Recall the well-known estimate $|G(\frac{a}{q},\ \ell)|\leq Cq^{-\frac{d}{2}}$, we have
\begin{equation}\label{20240929g}
    \|T_{g^{\frac{a}{q}}}f\|_p\leq Cq^{-\frac{d(p-1)}{p}}\|f\|_p,\quad 1\leq p\leq 2.\\
\end{equation}
For $d\ge5$ and $\frac{d}{d-2}<p\le2$, we conclude, by Lemma \ref{lem:scalar-maximal} together with the inequalities \eqref{20240929f} and \eqref{20240929g}, gives:
\begin{equation}\label{estimate N lambda a/q}
    \begin{aligned}
    &\big\|{\sup_{0<\lambda<\infty}}^+N_\lambda^\frac{a}{q} f\big\|_p\le\big\|{\sup\limits_{0< \lambda<\infty}}^{+}N^{\frac{a}{q}}_{\lambda}f\big\|_p\\
    &\qquad\leq {\sup\limits_{0<\lambda<\infty}}\Big|\frac{c_{d}\lambda^{d-2}}{r_{d}(\lambda^{2})}e^{-2\pi i\lambda^{2}\frac{a}{q}}\Big|\cdot\big\|{\sup\limits_{0<\lambda<\infty}}^{+}T_{(\psi_q\hat{\sigma}_\lambda)_\mathrm{per}}T_{g^{\frac{a}{q}}}f\big\|_p\\
     &\qquad\leq C\|T_{g^{\frac{a}{q}}}f\|_p\leq C q^{-\frac{d(p-1)}{p}}\|f\|_p.
    \end{aligned}
\end{equation}
Here the third inequality holds for $d\ge5$. By triangle inequality, we obtain
\begin{equation*}
    \begin{aligned}
        &\big\|{\sup_{0<\lambda<\infty}}^+N_\lambda f\big\|_p\le\sum_{\frac{a}{q}\in\mathcal F}\big\|{\sup\limits_{0< \lambda<\infty}}^{+}N^{\frac{a}{q}}_{\lambda}f\big\|_p\\
        &\qquad\le C\sum_{\frac{a}{q}\in\mathcal F}q^{-\frac{d(p-1)}{p}}\|f\|_p\le C\|f\|_p.
    \end{aligned}
\end{equation*}
Here the last inequality follows from the condition $p>\frac{d}{d-2}$.

\end{proof}

\section{\textbf{Proof of Theorem \ref{20240929h}}}\label{sec:maximal-proof}

Now we are ready to prove Theorem \ref{20240929h}.

\begin{proof}
Step 1. We claim that for \(d\ge 5\) and \(p>\frac{d}{d-2}\), there exists a constant \(C>0\) depending on $d$ and $p$ such that
\begin{equation*}
\big\|{\sup\limits_{0<\lambda<\infty}}^{+}M_{\lambda}f\big\|_{p}\leq C\|f\|_p,\quad \forall f\in L_{p}(\ell_{\infty}(\mathbb{Z}^{d})\Bar{\otimes}\mathcal{M})
\end{equation*}

Since Proposition \ref{20240929c} already provides the maximal
estimate for $N_\lambda$, it remains to establish the corresponding
estimate for $M_\lambda-N_\lambda$.
We begin with a local bound: Corollary \ref{20240929i} and
Proposition \ref{20240929c} give, uniformly in $\Lambda\ge1$,
\begin{equation}\label{20240929k}
\big\|{\sup\limits_{\Lambda\leq \lambda\leq 2\Lambda}}^{+}(M_{\lambda}-N_{\lambda})f\big\|_p\leq C\|f\|_p.
\end{equation}
For $p=2$, combining Proposition \ref{20240929j} with estimate \eqref{estimate N lambda a/q}, we obtain, for every $f\in L_2(\ell_\infty(\mathbb Z^d)\bar\otimes\mathcal M)$,
\begin{equation}\label{eq_L_2 approxi}
\begin{aligned}
&\big\|{\sup\limits_{\Lambda\leq \lambda\leq 2\Lambda}}^{+}(M_{\lambda}-N_{\lambda})f\big\|_{2}\\
 &\qquad\leq\sum_{\frac{a}{q}\in \mathcal{F}_{\Lambda}}
\big\|{\sup\limits_{\Lambda\leq \lambda\leq 2\Lambda}}^{+}(M^{\frac{a}{q}}_{\lambda}-N^{\frac{a}{q}}_{\lambda})f\big\|_{2}+\sum_{\frac{a}{q}\in \mathcal{F}\setminus \mathcal{F}_{\Lambda}}\big\|{\sup\limits_{\Lambda\leq \lambda\leq 2\Lambda}}^{+}N^{\frac{a}{q}}_{\lambda}f\big\|_{2}\\
&\qquad\leq C\Big(\Lambda^{2-\frac{d}{2}}+\sum_{\frac{a}{q}\in \mathcal{F}\setminus \mathcal{F}_{\Lambda}}q^{-\frac{d}{2}}\Big)\|f\|_{2}
\leq C\Lambda^{2-\frac{d}{2}}\|f\|_{2}.
\end{aligned}  
\end{equation}
Interpolating this with the inequality \eqref{20240929k} yields
\begin{equation*}
\big\|{\sup\limits_{\Lambda\leq \lambda\leq 2\Lambda}}^{+}(M_{\lambda}-N_{\lambda})f\big\|_p\leq C\Lambda^{(2-\frac{d}{2})\theta_{p}}\|f\|_p    
\end{equation*}
with $\frac{d}{d-2}<p\leq 2$ and $0<\theta_{p}\leq 1$.
Now observe that $\lambda\geq 1$, and thus
\begin{equation*}
\begin{aligned}
&\big\|{\sup\limits_{0<\lambda<\infty}}^{+}(M_{\lambda}-N_{\lambda})f\big\|_p
\leq \sum_{k\geq 0}\big\|{\sup\limits_{2^{k}\leq \lambda\leq 2^{k+1}}}^{+}(M_{\lambda}-N_{\lambda})f\big\|_p\\  
&\qquad\leq C\sum_{k\geq 0}2^{(2-\frac{d}{2})\theta_{p}k}\|f\|_p\leq C\|f\|_{L_{p}(\ell_{\infty}(\mathbb{Z}^{d})\bar{\otimes}\mathcal{M})}
.
\end{aligned}
\end{equation*}
For $\frac{d}{d-2}<p\leq 2$ and $d\ge5$,  combining this with the inequality \eqref{20240929l}, we conclude 
\begin{equation*}
\big\|{\sup\limits_{0<\lambda<\infty}}^{+}M_{\lambda}f\big\|_p
\leq C\|f\|_p,\quad \forall f\in L_{p}(\ell_{\infty}(\mathbb{Z}^{d})\bar{\otimes}\mathcal{M}). 
\end{equation*}
The corresponding estimate for $p=\infty$ is trivial, and hence the result for $p>2$ follows by interpolation, completing the proof of the claim.

\smallskip

Step 2. We now show how the maximal inequality of $M_{\lambda}$ on $L_{p}(\ell_{\infty}(\mathbb{Z}^{d})\bar{\otimes}\mathcal{M})$ leads to the maximal inequality of $M^{\gamma}_{\lambda}$ on $L_{p}(\mathcal{M})$.
Note that
$$\big\|{\sup\limits_{0<\lambda<\infty}}^{+}M_{\lambda}^{\gamma}x\big\|_p={\sup\limits_{N\geq 1}}\big\|{\sup\limits_{1\leq \lambda\leq N}}^{+}M_{\lambda}^{\gamma}x\big\|_p,\quad \forall x\in L_{p}(\mathcal{M}).$$
To prove Theorem \ref{20240929h}, it suffices to show that for any positive integer $N$, the following inequality
\begin{equation}\label{20240930e}
 \big\|{\sup\limits_{1\leq \lambda\leq N}}^{+}M_{\lambda}^{\gamma}x\big\|_p\leq C\|x\|_p,\quad \forall x\in \mathcal{S}_{\mathcal{M}}   
\end{equation}
holds for a positive constant $C$ independent of $N$.
Fix $x\in \mathcal{S}_{\mathcal{M}}$ and define a $\mathcal{S}_{\mathcal{M}}$-valued function on $\mathbb{Z}^{d}$ by $f(n):=\gamma^{n}x$, then
$$M_{\lambda}f(n)=\frac{1}{r_{d}(\lambda^{2})}\sum_{|m|=\lambda}f(n+m)=\frac{1}{r_{d}(\lambda^{2})}\sum_{|m|=\lambda}\gamma^{n+m}x=\gamma^{n}M^{\gamma}_{\lambda}x$$
for all $0<\lambda<\infty$. Here the last equality above holds since $\{\gamma_{i}\}$ is a commuting family.
For a large positive integer $J$, we define $g\in L_p(\ell_\infty(\mathbb{Z}^d)\Bar{\otimes}\mathcal{M})$ by
$g(n)=f(n)\chi_{|n|\leq J}.$  For all $1\leq \lambda\leq N$, it is easy to check that
\begin{equation}\label{20240930a}
M_{\lambda}g(n)=\gamma^{n}M^{\gamma}_{\lambda}x,\ \forall n\ \text{with}\ |n|\leq J-N.
\end{equation}
By the claim in Step 1, we get
\begin{equation}\label{20240930b}
\begin{aligned}
&\big\|{\sup\limits_{1\leq \lambda\leq N}}^{+}M_{\lambda}g\big\|^{p}_p\leq C\|g\|^{p}_{p}=C\sum_{|n|\leq J}\|\gamma^{n}x\|^{p}_{p}\leq CJ^{d}\|x\|^{p}_{p}.
\end{aligned}
\end{equation}
The last inequality follows since each $\gamma_i$ is an isometry on
$L_p(\mathcal M)$ and $\#\{n\in\Z^d:|n|\leq J\}\asymp_d J^d$. An application of Cauchy-Schwartz inequality (as shown in the proof of \cite[Theorem 3.1]{HLW21}) yields
\begin{equation}\label{20240930c}
   \sum_{n\in\mathbb{Z}^{d}} \big\|{\sup\limits_{1\leq \lambda\leq N}}^{+}M_{\lambda}g(n)\big\|^{p}_{L_{p}(\mathcal{M})}\leq \big\|{\sup\limits_{1\leq \lambda\leq N}}^{+}M_{\lambda}g\big\|^{p}_{L_{p}(\ell_{\infty}(\mathbb{Z}^{d})\Bar{\otimes}\mathcal{M})}.
\end{equation}
And it is easy to check by definition that
\begin{equation}\label{20240930d}
\big\|{\sup\limits_{1\leq \lambda\leq N}}^{+}\gamma^{n}M_{\lambda}^{\gamma}x\big\|_p=\big\|{\sup\limits_{1\leq \lambda\leq N}}^{+}M_{\lambda}^{\gamma}x\big\|_p.
\end{equation}
A combination of the formulae and inequalities \eqref{20240930a}, \eqref{20240930b}, \eqref{20240930c} and \eqref{20240930d} yields the following estimate
\begin{equation*}
\begin{aligned}
(J-N)^{d}\big\|{\sup\limits_{1\leq \lambda\leq N}}^{+}M_{\lambda}^{\gamma}x\big\|^{p}_{p}&\leq C\sum_{|n|\leq J-N}\big\|{\sup\limits_{1\leq \lambda\leq N}}^{+}\gamma^{n}M_{\lambda}^{\gamma}x\big\|^{p}_p\\
&\leq C\sum_{n\in\mathbb{Z}^{d}} \big\|{\sup\limits_{1\leq \lambda\leq N}}^{+}M_{\lambda}g(n)\big\|^{p}_{p}\\
&\leq CJ^{d}\|x\|^{p}_p.
\end{aligned}
\end{equation*}
That is,
\begin{equation*}
\big(\frac{J-N}{J}\big)^{\frac{d}{p}}\big\|{\sup\limits_{1\leq \lambda\leq N}}^{+}M_{\lambda}^{\gamma}x\big\|_{p}\leq C\|x\|_{p}.
\end{equation*}
A limiting argument on $J$ yields the estimate \eqref{20240930e}, completing the proof.
\end{proof}

\section{\textbf{Oscillation estimate and pointwise convergence}}\label{sec:oscillation}

In this section, we establish an oscillation estimate and use it
to prove Theorem \ref{thm:pointwise-main}. 

{Throughtout this section, we fix a noncommutative probability space $(\mathcal M,\tau)$ and strongly ergodic trace-preserving $*$-automorphisms $\gamma=(\gamma_1,\dotsm,\gamma_d)$ on it.}

\subsection{An oscillation estimate}\label{sec:oscillation-estimate}
 In this subsection, we fix $x\in L_\infty(\mathcal{M})$ with $\tau(x)=0$. Given a large number $L\geq1$, we define an operator valued function  $F_L\in L_2(\ell_\infty(\Z^d)\Bar{\otimes}\mathcal M)$ by $ F_L(n)=\gamma^n(x)\chi_{B_L}(n),$ where $B_L=[-L,L]^d\cap\Z^d.$ Let $(N_k)_{k\geq1}$ be a rapidly increasing positive integer sequence satisfying 
 \begin{equation}\label{speed Nk}
     N_{k+1}>N_k^{100}. 
 \end{equation}
 
 Define the {generalized dyadic} blocks $I_k=\{\lambda>0:\lambda^2\in\mathbb N,\ N_k\leq\lambda^2<N_{k+1}\},$ and the associated oscillation quantity
\begin{equation*}
    \mathscr O_{K,L}(x)=\frac{1}{K|B_L|}\sum_{k=1}^K
 \big\|{\sup\limits_{\lambda\in I_k}}^{+}M_\lambda F_L\big\|_
 2^2.
\end{equation*}

Our main result of this subsection is the following Proposition.
\begin{proposition}\label{prop:oscillation}
If $d\geq 5$, then for every rapidly increasing sequence $(N_k)_{k\geq 1}$ of
positive integers satisfying \eqref{speed Nk}, we have
\begin{equation*}
\lim_{K\to\infty}\limsup_{L\to\infty}
\mathscr O_{K,L}(x)=0.
\end{equation*}
\end{proposition}
Before giving this proof, we need to do some preparation. Recall that $N_\lambda=\sum_{\frac{a}{q}\in\mathcal{F}}N_\lambda^\frac{a}{q}$. Fix integer $\Lambda>0$ to be chosen later, we define $N_\lambda^{\le \Lambda}=\sum_{\frac{a}{q}\in\mathcal{F}_\Lambda}N_\lambda^\frac{a}{q}$ and $N_\lambda^{>\Lambda}=\sum_{\frac{a}{q}\in\mathcal{F}\setminus \mathcal{F}_\Lambda}N_\lambda^\frac{a}{q}$.
Thus
\begin{equation}\label{eq:decomposition_for_osci}
     M_\lambda F_L=(M_\lambda-N_\lambda)F_L
                  +N_\lambda^{>\Lambda}F_L+N_\lambda^{\leq\Lambda}F_L.
\end{equation}

The following proposition provides an estimate for the oscillation of the last term.




\begin{lemma}\label{lem:finite-denominator}
Under the assumptions of Proposition \ref{prop:oscillation}, for every rapidly increasing sequence $(N_k)_{k\geq 1}$ of
positive integers satisfying \eqref{speed Nk} and every fixed integer $\Lambda>0$, we have:
\begin{equation*}
 \lim_{K\to\infty}\limsup_{L\to\infty}
 \frac{1}{K|B_L|}\sum_{k=1}^K
 \big\|{\sup\limits_{\lambda\in I_k}}^{+}N_\lambda^{\leq\Lambda}F_L\big\|_
 2^2=0.
\end{equation*}
\end{lemma}

Lemma \ref{lem:finite-denominator} constitutes the key ingredient in proving Proposition \ref{prop:oscillation} and thus the pointwise convergence. The whole Section \ref{sec:finite-denominator-proof} will be devoted to its proof. Here we use the two estimates \eqref{eq_L_2 approxi} and \eqref{estimate N lambda a/q} to prove Proposition \ref{prop:oscillation}.

\begin{proof}[Proof of Proposition \ref{prop:oscillation}]

Applying estimate \eqref{eq_L_2 approxi}, we obtain
\begin{equation}\label{eq:error-tail}
\begin{aligned}
&\big\|{\sup\limits_{\lambda\in I_k}}^{+}
(M_\lambda-N_\lambda)F_L \big\|_2
\leq \sum_{j\geq0}
\big\|{\sup\limits_{2^j\sqrt{N_k}\leq\lambda\leq2^{j+1}\sqrt{N_k}}}^{+}
(M_\lambda-N_\lambda)F_L\big\|_2\\
&\qquad\leq C\sum_{j\geq0}
\big(2^j\sqrt{N_k}\big)^{2-\frac{d}{2}}\|F_L\|_2
\leq CN_k^{1-\frac{d}{4}}\|F_L\|_2.
\end{aligned}
\end{equation}
Similarly, applying estimate \eqref{estimate N lambda a/q} at $p=2$, we have
\begin{equation}\label{eq:q-tail}
\begin{aligned}
    \big\|{\sup\limits_{0<\lambda<\infty}}^{+}N_\lambda^{> \Lambda}F_L\big\|_2\leq \sum_{\frac{a}{q}\in\mathcal{F}\setminus\mathcal{F}_\Lambda}Cq^{-\frac{d}{2}}\|F_L\|_2
    \le C\Lambda^{2-\frac{d}{2}}\|F_L\|_2. 
\end{aligned}
\end{equation}
Since $\|F_L\|_2^2=|B_L|\|x\|_2^2$, combining the decomposition \eqref{eq:decomposition_for_osci} with the estimates \eqref{eq:error-tail} and \eqref{eq:q-tail}, we obtain
\begin{equation*}
 \begin{aligned}
 &\mathscr O_{K,L}(x)\leq \frac{C}{K}\sum_{k=1}^K N_k^{2-\frac{d}{2}}\|x\|_2^2
       +C\Lambda^{4-d}\|x\|_2^2+\frac{C}{K|B_L|}\sum_{k=1}^K
 \big\|{\sup\limits_{\lambda\in I_k}}^{+}N_\lambda^{\le\Lambda}F_L\big\|_
 2^2.
 \end{aligned}
\end{equation*}
For fixed $\Lambda$, Lemma \ref{lem:finite-denominator} and $d\geq5$ give
\[
 \limsup_{K\to\infty}\limsup_{L\to\infty}\mathscr O_{K,L}(x)
 \leq C_d\Lambda^{4-d}\|x\|_2^2.
\]
Letting $\Lambda\to\infty$ completes the proof.
\end{proof}

\subsection{Proof of Theorem \ref{thm:pointwise-main}}\label{sec:convergence}
We deduce decay of the tail $L_2(\ell_\infty)$-norm from Proposition \ref{prop:oscillation} by contradiction: failure of tail decay produces arbitrary rapidly separated blocks with a uniform positive lower bound, contrary to Proposition \ref{prop:oscillation}. We first record the transference estimate needed for this argument.

\begin{lemma}\label{lem:bad-transfer}
Fix $K$ and let $L>\sqrt{N_{K+1}}$.  Then for every $k\leq K$,
\begin{align*}
&\big\|{\sup\limits_{\lambda\in I_k}}^{+}M_\lambda^\gamma x\big\|_2^2\le\frac{1}{\big|B_{L-\sqrt{N_{K+1}}}\big|}\big\|{\sup\limits_{\lambda\in I_k}}^{+}M_\lambda F_L\big\|_
 {L_{2}(\ell_{\infty}(\Z^d)\Bar{\otimes}\mathcal{M})}^2, \quad\forall x\in L_{2}(\mathcal{M}).
\end{align*}
\end{lemma}

\begin{proof}
If $n\in B_{L-\sqrt{N_{K+1}}}$, $\lambda\in I_k$, and $k\leq K$, then
$n-m\in B_L$ whenever $|m|=\lambda$, since
$|m|=\lambda<\sqrt{N_{K+1}}$.  Therefore, 
\begin{align*}
 &M_\lambda F_L(n)=\frac{1}{r_d(\lambda^2)}\sum_{|m|=\lambda}\gamma^{n-m}(x)=\gamma^n\Big(\frac{1}{r_d(\lambda^2)}
                   \sum_{|m|=\lambda}\gamma^{-m}(x)\Big)
 =\gamma^n(M_\lambda^\gamma x).
\end{align*}
An argument similar to Section \ref{sec:maximal-proof} yields:
\begin{equation*}
\begin{aligned}
&\big|B_{L-\sqrt{N_{K+1}}}\big|
\big\|{\sup\limits_{\lambda\in I_k}}^{+}
M_\lambda^\gamma x\big\|_2^2
=\sum_{n\in B_{L-\sqrt{N_{K+1}}}}
\big\|{\sup\limits_{\lambda\in I_k}}^{+}
\gamma^n(M_\lambda^\gamma x)\big\|_2^2\\
&\qquad\leq\sum_{n\in\Z^d}
\big\|{\sup\limits_{\lambda\in I_k}}^{+}
M_\lambda F_L(n)\big\|_{L_2(\mathcal{M})}^2
\leq\big\|{\sup\limits_{\lambda\in I_k}}^{+}
M_\lambda F_L\big\|_{L_{2}(\ell_{\infty}(\Z^d)\Bar{\otimes}\mathcal{M})}^2.
\end{aligned}
\end{equation*}
\end{proof}

\begin{proof}[Proof of Theorem \ref{thm:pointwise-main}]
We first take $x\in L_\infty(\mathcal M)$ and put
$x_0=x-\tau(x)1$.  By Theorem \ref{20240929h} with $p=2$,
\begin{equation*}
 \big\|{\sup\limits_{\lambda^2\in\mathbb N}}^{+}M_\lambda^\gamma x_0\big\|_
 2
 \leq C_d\|x_0\|_2.
\end{equation*}
Define $d_R=\big\|{\sup\limits_{\lambda\geq R}}^{+}M_\lambda^\gamma x_0\big\|_2.$
Suppose that \(d_R\not\to 0\) as \(R\to\infty\). Then, by \eqref{eq:finite to infinite}, there exist \(\delta>0\) and a sufficiently rapidly increasing sequence \((N_k)_{k\ge1}\) of positive integers satisfying \(N_{k+1}>N_k^{100}\) such that the following estimate holds:
\[
 \big\|{\sup\limits_{\lambda\in I_k}}^{+}M_\lambda^\gamma x_0\big\|_
 2\geq\frac{\delta}{2}.
\]
Let $F_L(n)=\gamma^n(x_0)\chi_{B_L}(n)$. For $L>\sqrt{N_{K+1}}$, Lemma \ref{lem:bad-transfer} gives,
\begin{equation*}
\begin{aligned}
&\mathscr O_{K,L}(x_0)
=\frac{1}{K|B_L|}\sum_{k=1}^K
\big\|{\sup\limits_{\lambda\in I_k}}^{+}
M_\lambda F_L\big\|_{L_{2}(\ell_{\infty}(\Z^d)\Bar{\otimes}\mathcal{M})}^2\\
&\qquad\geq
\frac{\big|B_{L-\sqrt{N_{K+1}}}\big|}{K|B_L|}
\sum_{k=1}^K
\big\|{\sup\limits_{\lambda\in I_k}}^{+}
M_\lambda^\gamma x_0\big\|_2^2
\geq\frac{\delta^2}{4}
\frac{\big|B_{L-\sqrt{N_{K+1}}}\big|}{|B_L|}.
\end{aligned}
\end{equation*}
Since $\frac{\big|B_{L-\sqrt{N_{K+1}}}\big|}{|B_L|}\to1$ for fixed $K$ as $L\to\infty$,
Proposition \ref{prop:oscillation} yields the contradiction
\[
 0<\frac{\delta^2}{4}
 \leq\lim_{K\to\infty}\limsup_{L\to\infty}\mathscr O_{K,L}(x_0)=0.
\]
Thus $d_R\to0$. By equation \eqref{eq:tail Lp}, it follows that $(M_\lambda^\gamma x_0)_{\lambda^2\in\mathbb N}$ belongs to $L_2(\mathcal M;c_0)$. Consequently,
\[
 M_\lambda^\gamma x-\tau(x)1=M_\lambda^\gamma x_0
 \longrightarrow0\qquad\mathrm{b.a.u.}.
\]

Now let $\frac{d}{d-2}<p<\infty$.  The subspace
$L_\infty(\mathcal M)$ is dense in $L_p(\mathcal M)$.
The maximal estimate in Theorem \ref{20240929h}, together with the standard Banach principle for b.a.u. convergence {(see e.g.\cite{JX07,HongWang26})}, 
therefore extends the conclusion from $L_\infty(\mathcal{M})$ to every $L_p(\mathcal M)$. 
\end{proof}

\section{\textbf{Proof of Lemma \ref{lem:finite-denominator}}}\label{sec:finite-denominator-proof}
In this section, we prove Lemma \ref{lem:finite-denominator} by decomposing $N_\lambda^{\frac{a}{q}}$ into four parts in \eqref{eq:four-piece-decomposition}, and estimating the oscillation of each part separately. 
We begin with some preparations. For $\frac{a}{q}\in\mathcal{F}_\Lambda$, recall that the operator $N_\lambda^{\frac{a}{q}}$ has symbol
\[
   n_\lambda^{\frac{a}{q}}(\xi)= \frac{c_{d}\lambda^{d-2}}{r_{d}(\lambda^{2})}e^{-2\pi i\lambda^{2}\frac{a}{q}}\sum_{\ell\in \mathbb{Z}^{d}}G\big(\frac{a}{q},\ \ell\big)\phi_{q}\big(\xi-\frac{\ell}{q}\big)\hat{\sigma}_{\lambda}\big(\xi-\frac{\ell}{q}\big),
\]
where $\xi\in \mathbb{R}^d$. For each integer $k\geq1$, we choose a smooth partition of unity
$(w_{k,i})_{i=1}^3$ satisfying:
\begin{align*}
 {\mathrm{supp}}\; w_{k,1}
 &\subset\{\xi\in \mathbb R^d:\;|\xi|\leq N_{k+1}^{-1}\},\\
 \mathrm{supp}\; w_{k,2}
 &\subset\big\{\xi\in \mathbb{R}^d:\frac{1}{2}N_{k+1}^{-1}\leq|\xi|
                    \leq2N_k^{-\frac{1}{4}}\big\},\\
 \mathrm{supp}\; w_{k,3}
 &\subset\{\xi\in \mathbb{R}^d:|\xi|\geq N_k^{-\frac{1}{4}}\}.
\end{align*}
Exploiting the partition of identity, we obtain $n_{\lambda}^\frac{a}{q}(\xi)=\sum_{i=1}^3n_{\lambda,k}^{\frac{a}{q},i}(\xi),$ where
\begin{equation*}
 n_{\lambda,k}^{\frac{a}{q},i}(\xi)
 =\frac{c_{d}\lambda^{d-2}}{r_{d}(\lambda^{2})}e^{-2\pi i\lambda^{2}\frac{a}{q}}\sum_{\ell\in\Z^d}G\big(\frac{a}{q},\ell\big)\phi_q\big(\xi-\frac{\ell}{q}\big)
       w_{k,i}\big(\xi-\frac{\ell}{q}\big)
       \widehat\sigma_\lambda\big(\xi-\frac{\ell}{q}\big).
\end{equation*}
Define the Fourier multiplier $T_{\psi_{q,k}^i}$ on functions on $\mathbb Z^d$ with symbol
\begin{equation*}
 {\psi_{q,k}^i}(\xi)
 =\sum_{\ell\in\Z^d}\psi_q\big(\xi-\frac{\ell}{q}\big)w_{k,i}\big(\xi-\frac{\ell}{q}\big).
\end{equation*}
Recall that $\psi\in C_c^\infty(\mathbb{R}^d)$ is supported in $Q$ with $\psi=1$ on the support of $\phi$. Hence, we obtain, for $f\in L_{2}(\ell_{\infty}(\Z^d)\Bar{\otimes}\mathcal{M})$:  
\begin{equation}\label{eq:piece-factorization}
 N_{\lambda,k}^{\frac{a}{q},i}f
 =\frac{c_{d}\lambda^{d-2}}{r_{d}(\lambda^{2})}e^{-2\pi i\lambda^{2}\frac{a}{q}}T_{(\psi_q\widehat\sigma_{\lambda})_{\mathrm{per}}}
   T_{\psi_{q,k}^i}T_{g^{\frac{a}{q}}}f.
\end{equation}
Define the error term of low-frequency operator by $$E_{\lambda,k}^{\frac{a}{q}}f=N_{\lambda,k}^{\frac{a}{q},1}f-\frac{c_{d}\lambda^{d-2}}{r_{d}(\lambda^{2})}e^{-2\pi i\lambda^{2}\frac{a}{q}}T_{{\psi_{q,k}^1}}T_{g^{\frac{a}{q}}}f.$$ The four terms to be estimated are therefore,
\begin{equation}\label{eq:four-piece-decomposition}
 \begin{split}
 N_\lambda^{\frac{a}{q}}f=E_{\lambda,k}^{\frac{a}{q}}f+N_{\lambda,k}^{\frac{a}{q},2}f+N_{\lambda,k}^{\frac{a}{q},3}f+\frac{c_{d}\lambda^{d-2}}{r_{d}(\lambda^{2})}e^{-2\pi i\lambda^{2}\frac{a}{q}}T_{{\psi_{q,k}^1}}T_{g^{\frac{a}{q}}}f.
 \end{split}
\end{equation}
{They are called the low-frequency error term,  the middle-frequency term, the high-frequency
term, and the
low-frequency
 main term respectively.} {We will estimate these four terms in a general measure-theoretic setting, without assuming that the underlying noncommutative measure space is a probability space. The first three estimates will be established for arbitrary functions $f\in L_{2}(\ell_{\infty}(\Z^d)\Bar{\otimes}\mathcal{M})$, whereas the fourth will be proved only for functions of a particular form $F_L: n\rightarrow \gamma^n(x)\chi_{B_L}(n)$ where  the assumption of {\it strong ergodicity} is relevant. }
 
 Now we estimate the oscillation of {the low-frequency error term.}

\begin{lemma}\label{lem:low-error}
{Let $d\ge 5$.}
For every $k>0$ and every reduced fraction $\frac{a}{q}\in\mathcal{F}$, we have: 
\begin{equation*}
 \big\|{\sup\limits_{\lambda\in I_k}}^{+}E_{\lambda,k}^{\frac{a}{q}}f\big\|_
 2
 \leq Cq^{-\frac{d}{2}}N_{k+1}^{-\frac{1}{2}}\|f\|_2, \quad \forall f\in L_{2}(\ell_{\infty}(\Z^d)\Bar{\otimes}\mathcal{M}).
\end{equation*}
\end{lemma}

\begin{proof}
By the rotation invariance of normalized spherical measure, we have $|\widehat\sigma(\xi)-1|\leq C|\xi|^2, \ \xi\in \mathbb{R}^d.$ For $\lambda\in I_k$ and $\eta\in\mathrm{supp}\; w_{k,1}$, we have
$\lambda|\eta|\leq N_{k+1}^{-\frac{1}{2}}$, and hence
$|\widehat\sigma_\lambda(\eta)-1|\leq CN_{k+1}^{-1}$.
Combining this estimate with the disjointness of the
translates of $\mathrm{supp}\;\psi_q$, we obtain
\[
 \sup\limits_{\lambda\in I_k}
 \Big\|
 \sum_{\ell\in\Z^d}
 \psi_q\big(\xi-\frac{\ell}{q}\big)w_{k,1}\big(\xi-\frac{\ell}{q}\big)
 \big(\widehat\sigma_\lambda\big(\xi-\frac{\ell}{q}\big)-1\big)
 \Big\|_{L_\infty(\mathbb{T}^d)}
 \leq CN_{k+1}^{-1}.
\]
Applying Lemma \ref{lem:L2-maximal-sum} and Plancherel's theorem,
we obtain
\begin{align*}
 & \big\|{\sup\limits_{\lambda\in I_k}}^{+}E_{\lambda,k}^{\frac{a}{q}}f\big\|_
 2\leq
 C\Big(\sum_{\lambda\in I_k}
       \|E_{\lambda,k}^{\frac{a}{q}}f\|_2^2\Big)^{\frac{1}{2}}\\
 &\qquad\leq C(\#I_k)^{\frac{1}{2}}N_{k+1}^{-1}
       \|T_{g^{\frac{a}{q}}}f\|_2\leq Cq^{-\frac{d}{2}}N_{k+1}^{-\frac{1}{2}}\|f\|_2.
\end{align*}
Here we used $\#I_k\leq N_{k+1}$ and the decay of Gauss sum \eqref{20240929g}.
\end{proof}

Next we estimate the oscillation of the middle frequency term.
\begin{lemma}\label{lem:middle}
Let $d\ge5$. For every $K\geq1$ and every reduced fraction $\frac{a}{q}\in \mathcal{F}$, we have:
\begin{equation*}
 \sum_{k=1}^K
 \big\|{\sup\limits_{\lambda\in I_k}}^{+}N_{\lambda,k}^{\frac{a}{q},2}f\big\|_
 {2}^2
 \leq Cq^{-d}\|f\|_2^2, \quad \forall f\in L_{2}(\ell_{\infty}(\Z^d)\Bar{\otimes}\mathcal{M}).
\end{equation*}
\end{lemma}

\begin{proof}
By the estimate \eqref{20240929f}, if $d\ge3$, we have:
\begin{equation*}
 \big\|{\sup\limits_{0<\lambda<\infty}}^{+}T_{(\psi_q\widehat{\sigma}_{\lambda})_{\mathrm{per}}}f\big\|_
 {2}
 \leq C\|f\|_{2}, \quad\forall f\in L_2(\ell_\infty(\Z^d)\Bar{\otimes}\mathcal M).
\end{equation*}
By the rapidly increasing property of $(N_k)_{k\ge1}$, we have $2N_{k+2}^{-\frac{1}{4}}<\frac{1}{2}N_{k+1}^{-1}$. So the annuli 
$\{\xi\in\mathbb{R}^d:\frac{1}{2}N_{k+1}^{-1}\leq|\xi|\leq2N_k^{-\frac{1}{4}}\}$ overlap at most twice.
Since the translates of $\mathrm{supp}\;\psi_q$ are disjoint, we have $\sum_{k\geq1}|\psi^2_{q,k}(\xi)|^2\leq2.$
Thus, the preceding maximal estimate, \eqref{eq:piece-factorization}, Lemma \ref{lem:scalar-maximal}, and the Plancherel theorem yield
\begin{align*}
 &\sum_{k=1}^K
 \big\|{\sup\limits_{\lambda\in I_k}}^{+}N_{\lambda,k}^{\frac{a}{q},2}f\big\|_
 2^2\leq\sum_{k=1}^K\sup_{\lambda\in I_k}\Big|\frac{c_{d}\lambda^{d-2}}{r_{d}(\lambda^{2})}e^{-2\pi i\lambda^{2}\frac{a}{q}}\Big|^2\cdot\big\|{\sup_{\lambda\in I_k}}^+T_{(\psi_q\widehat\sigma_{\lambda})_{\mathrm{per}}}
   T_{{\psi_{q,k}^2}}T_{g^{\frac{a}{q}}}f\big\|_2^2\\
 &\qquad\le C\sum_{k=1}^K
       \|T_{{\psi_{q,k}^2}}T_{g^{\frac{a}{q}}}f\|_2^2\leq C\|T_{g^{\frac{a}{q}}}f\|_2^2\leq Cq^{-d}\|f\|_2^2,
\end{align*}
where the last inequality comes from the decay of Gauss sum \eqref{20240929g}.
\end{proof}
To estimate the high-frequency term, we first establish the following
maximal inequality for continuous spherical averages. Let \(P_N\) denote the Fourier multiplier on \(\mathbb{R}^d\) with a smooth symbol \(p_N\in C^\infty(\mathbb{R}^d)\) satisfying
\[
\operatorname{supp}p_N\subseteq\{\xi\in\mathbb{R}^d:|\xi|\ge N\},
\quad
p_N(\xi)=1\;\text{whenever }|\xi|\ge 2N.
\] 

\begin{lemma}\label{lem:continuous-high-frequency}
Let $d\geq3$. For $R,N>0$, we have:
$$
\big\|{\sup\limits_{t\geq R}}^{+}\sigma_t*P_Nf\big\|_2
\leq C(RN)^{-\frac{d-2}{2}}\|f\|_2,\quad \forall f\in L_2(L_\infty(\mathbb R^d)\Bar{\otimes}\mathcal M)
$$
where $C$ depends only on $d$.
\end{lemma}
\begin{proof}
   The estimate follows from a more general result in \cite[Lemma 3.1]{HLSW}.
\end{proof}

\begin{lemma}\label{lem:high}
Let $d\ge5$. For every $k\geq1$ and every reduced fraction $\frac{a}{q}\in\mathcal{F}$, we have:
\begin{equation*}
 \big\|{\sup\limits_{\lambda\in I_k}}^{+}N_{\lambda,k}^{\frac{a}{q},3}f\big\|_
 {2}
 \leq Cq^{-\frac{d}{2}}N_k^{-\frac{d-2}{8}}\|f\|_2, \quad \forall f\in L_{2}(\ell_{\infty}(\Z^d)\Bar{\otimes}\mathcal{M}).
\end{equation*}
\end{lemma}

\begin{proof}
Combining the Lemma \ref{lem:continuous-high-frequency} with the support condition $\operatorname{supp}(\psi w_{{k,3}}(q^{-1}\cdot))
\subset Q\cap\{|\xi|\geq qN_k^{-\frac{1}{4}}\},$ we obtain:
\[
 \big\|{\sup\limits_{\lambda\geq \sqrt{N_k}}}^{+}\mathcal{T}_{\psi w_{k,3}(q^{-1}\cdot)\hat{\sigma}_{q^{-1}\lambda}}g\big\|_
 2
 \leq CN_k^{-\frac{d-2}{8}}\|g\|_2, \quad \forall g\in {L_{2}(L_\infty(\mathbb R^d)\Bar{\otimes}\mathcal{M})}.
\]
Applying the sampling Lemma \ref{20240929d} gives:
\[
 \big\|{\sup\limits_{\lambda\in I_k}}^{+}T_{(\psi_q w_{k,3}\widehat\sigma_{\lambda})_{\mathrm{per}}}f\big\|_2
 \leq CN_k^{-\frac{d-2}{8}}\|f\|_2, \quad \forall f\in L_{2}(\ell_{\infty}(\Z^d)\Bar{\otimes}\mathcal{M}).
\]
Since $\psi_q=1$ on $\operatorname{supp}\phi_q$, we have:
\[
 N_{\lambda,k}^{\frac{a}{q},3}f
 =\frac{c_{d}\lambda^{d-2}}{r_{d}(\lambda^{2})}e^{-2\pi i\lambda^{2}\frac{a}{q}}
   T_{(\psi_q w_{k,3}\widehat\sigma_{\lambda})_{\mathrm{per}}}
   T_{g^{\frac{a}{q}}}f.
\]
Combining this identity with Lemma \ref{lem:scalar-maximal}
and the preceding maximal estimate, we obtain:
\begin{align*}
 &\big\|{\sup\limits_{\lambda\in I_k}}^{+}N_{\lambda,k}^{\frac{a}{q},3}f\big\|_
 {2}\le\sup_{\lambda\in I_k}\Big|\frac{c_{d}\lambda^{d-2}}{r_{d}(\lambda^{2})}e^{-2\pi i\lambda^{2}\frac{a}{q}}\Big|\cdot\big\|{\sup_{\lambda\in I_k}}^+T_{(\psi_q w_{k,3}\widehat\sigma_{\lambda})_{\mathrm{per}}}
   T_{g^{\frac{a}{q}}}f\big\|_2\\
 &\qquad\leq CN_k^{-\frac{d-2}{8}}\|T_{g^{\frac{a}{q}}}f\|_2\leq Cq^{-\frac{d}{2}}N_k^{-\frac{d-2}{8}}\|f\|_2,
\end{align*}
where the last inequality comes from the decay of Gauss sum \eqref{20240929g}.
\end{proof}

Finally we estimate the oscillation of the low-frequency main term. Let $ \mathcal Q_q=(q^{-1}\Z^d)/\Z^d$ and define $ V_{q,k}=\{\theta\in\mathbb T^d:
       \operatorname{dist}_{\mathbb T^d}(\theta,\mathcal Q_q)\leq N_{k}^{-1}\},$
where the distance is induced by the Euclidean norm. Recall $F_L(n)=\gamma^n(x)\chi_{B_L}(n)$.

\begin{lemma}\label{lem:low-main}
Let $d\geq5$. For every $k\geq1$, and every reduced fraction $\frac{a}{q}\in\mathcal{F}$,  we have
\begin{equation*}
 \begin{split}
 &\limsup_{L\to\infty}\frac{1}{|B_L|}
 \Big\|{\sup\limits_{\lambda\in I_k}}^{+}\frac{c_{d}\lambda^{d-2}}{r_{d}(\lambda^{2})}e^{-2\pi i\lambda^{2}\frac{a}{q}}T_{\psi^{1}_{q,k}}T_{g^{\frac{a}{q}}}F_L\Big\|_
 2^2\leq Cq^{-d}\nu_x(V_{q,k}),
 \end{split}
\end{equation*}
where $\nu_x$ is the scalar spectral measure on $\mathbb{T}^d$ induced by $x$ defined in \eqref{eq:spectral-measure}. Consequently, if $\nu_x(\mathcal Q_q)=0$, then
\begin{equation*}
 \lim_{K\to\infty}\limsup_{L\to\infty}\frac{1}{K|B_L|}
 \sum_{k=1}^K
 \Big\|{\sup\limits_{\lambda\in I_k}}^{+}\frac{c_{d}\lambda^{d-2}}{r_{d}(\lambda^{2})}e^{-2\pi i\lambda^{2}\frac{a}{q}}T_{\psi^{1}_{q,k}}T_{g^{\frac{a}{q}}}F_L\Big\|_
 2^2=0.
\end{equation*}
\end{lemma}

\begin{proof}
Since $d\geq5$, $\frac{c_{d}\lambda^{d-2}}{r_{d}(\lambda^{2})}$ is uniformly bounded. Applying Lemma \ref{lem:scalar-maximal}, we have:
\begin{equation*}
 \Big\|{\sup\limits_{\lambda\in I_k}}^{+}\frac{c_{d}\lambda^{d-2}}{r_{d}(\lambda^{2})}e^{-2\pi i\lambda^{2}\frac{a}{q}}T_{\psi^{1}_{q,k}}T_{g^{\frac{a}{q}}}F_L\Big\|_
 2
 \leq C\|T_{\psi^{1}_{q,k}}T_{g^{\frac{a}{q}}}F_L\|_2.
\end{equation*}
Define $ \mathcal B_L(\theta)=\frac{1}{|B_L|}
       \Big|\sum_{n\in B_L}e^{2\pi in\cdot\theta}\Big|^2$ and denote $\psi_k^\frac{a}{q}(\xi)=|\psi_{q,k}^1(\xi)g^{\frac{a}{q}}(\xi)|^2$.
%
We claim the following identity holds:
\begin{equation*}
 \frac{1}{|B_L|}\|T_{\psi^1_{q,k}}T_{g^{\frac{a}{q}}}F_L\|_2^2
 =\int_{\mathbb T^d}(\mathcal B_L*\psi_k^\frac{a}{q})(\theta)
                         \,\,\mathrm{d}\nu_x(\theta).
\end{equation*}
By the definition of $\mathcal B_L$, its Fourier coefficients satisfy $\widehat{\mathcal B_L}(m)
 =\int_{\mathbb T^d}\mathcal B_L(\theta)
                  e^{-2\pi im\cdot\theta}\,\,\mathrm{d}\theta
 =\frac{A_m}{|B_L|},$ where $ A_m=\#\{(n,n')\in B_L\times B_L:n-n'=m\}, \ m\in\Z^d.$ By Plancherel's theorem and the spectral representation of $\gamma$,
we obtain
\begin{equation*}
\begin{aligned}
 &\frac{1}{|B_L|}
 \|T_{\psi^1_{q,k}}T_{g^{\frac{a}{q}}}F_L\|_2^2
 =\frac{1}{|B_L|}\int_{\mathbb T^d}
 \psi_k^\frac{a}{q}(\xi)\|\widehat F_L(\xi)\|_2^2
 \,\mathrm{d}\xi\\
 &\qquad=\frac{1}{|B_L|}\int_{\mathbb T^d}
 \psi_k^\frac{a}{q}(\xi)
 \Big\|\sum_{n\in B_L}\gamma^n(x)
 e^{-2\pi in\cdot\xi}\Big\|_2^2\,\mathrm{d}\xi\\
 &\qquad=\frac{1}{|B_L|}\sum_{n,n'\in B_L}
 \langle\gamma^n(x),\gamma^{n'}(x)\rangle
 \int_{\mathbb T^d}\psi_k^\frac{a}{q}(\xi)
 e^{-2\pi i(n-n')\cdot\xi}\,\mathrm{d}\xi\\
 &\qquad=\frac{1}{|B_L|}
 \int_{\mathbb T^d}\int_{\mathbb T^d}
 \psi_k^\frac{a}{q}(\xi)
 \Big(\sum_{n,n'\in B_L}
 e^{2\pi i(n-n')\cdot(\theta-\xi)}\Big)
 \,\mathrm{d}\xi\,\mathrm{d}\nu_x(\theta)\\
 &\qquad=\int_{\mathbb T^d}\int_{\mathbb T^d}
 \psi_k^\frac{a}{q}(\xi)
 \Big(\frac{1}{|B_L|}\sum_{m\in\Z^d}
 A_m e^{2\pi im\cdot(\theta-\xi)}\Big)
 \,\mathrm{d}\xi\,\mathrm{d}\nu_x(\theta)\\
 &\qquad=\int_{\mathbb T^d}\int_{\mathbb T^d}
 \psi_k^\frac{a}{q}(\xi)
 \Big(\sum_{m\in\Z^d}\widehat{\mathcal B_L}(m)
 e^{2\pi im\cdot(\theta-\xi)}\Big)
 \,\mathrm{d}\xi\,\mathrm{d}\nu_x(\theta)\\
 &\qquad=\int_{\mathbb T^d}\int_{\mathbb T^d}
 \psi_k^\frac{a}{q}(\xi)\mathcal B_L(\theta-\xi)
 \,\mathrm{d}\xi\,\mathrm{d}\nu_x(\theta)=\int_{\mathbb T^d}
 (\mathcal B_L*\psi_k^\frac{a}{q})(\theta)
 \,\mathrm{d}\nu_x(\theta).
\end{aligned}
\end{equation*}
The kernels $\mathcal B_L$ form a nonnegative approximate identity.  Thus, for fixed $q,a,k$,
$\mathcal B_L*\psi_k^\frac{a}{q}\to \psi_k^\frac{a}{q}$ uniformly. On the other hand, we have
$0\leq \psi_k^{\frac aq}\leq Cq^{-d}\mathbf{1}_{V_{q,k}}(\xi)$. Therefore, we obtain
\[
 \lim_{L\to\infty}\frac{1}{|B_L|}
       \|T_{\psi^1_{q,k}}T_{g^{\frac{a}{q}}}F_L\|_2^2
 =\int_{\mathbb T^d}\psi_k^\frac{a}{q}\,\,\mathrm{d}\nu_x
 \leq Cq^{-d}\nu_x(V_{q,k}),
\]
which completes the proof.
\end{proof}

Now we are ready to show Lemma \ref{lem:finite-denominator}.

\begin{proof}[Proof of Lemma \ref{lem:finite-denominator}]
Let $x\in L_\infty(\mathcal M)$ with  $\tau(x)=0$. By strong ergodicity, Lemma \ref{lem:rational-atoms} gives
$\nu_x(\mathcal Q_q)=0$ for each $q\leq\Lambda$.
Combining \eqref{eq:four-piece-decomposition} with Lemmas \ref{lem:low-error}, \ref{lem:middle}, \ref{lem:high}, and \ref{lem:low-main}, and using the identity $\|F_L\|_2^2=|B_L|\|x\|_2^2$, we obtain
\begin{equation*}
\begin{aligned}
 &\Big(\limsup_{L\to\infty}\frac{1}{K|B_L|}
 \sum_{k=1}^K
 \big\|{\sup\limits_{\lambda\in I_k}}^{+}
 N_\lambda^{\leq\Lambda}F_L\big\|_2^2\Big)^{\frac{1}{2}}\leq
 \sum_{\frac{a}{q}\in\mathcal F_\Lambda}
 \Big(\limsup_{L\to\infty}\frac{1}{K|B_L|}
 \sum_{k=1}^K
 \big\|{\sup\limits_{\lambda\in I_k}}^{+}
 N_\lambda^{\frac{a}{q}}F_L\big\|_2^2\Big)^{\frac{1}{2}}\\
 &\qquad\leq
 C\sum_{\frac{a}{q}\in\mathcal F_\Lambda}q^{-\frac{d}{2}}
 \Big[
 \frac{\|x\|_2^2}{K}
 \Big(1+\sum_{k=1}^K
 (N_{k+1}^{-1}+N_k^{-\frac{d-2}{4}})\Big)+
 \frac{1}{K}\sum_{k=1}^K\nu_x(V_{q,k})
 \Big]^{\frac{1}{2}}
 \longrightarrow0
 \quad(K\to\infty).
\end{aligned}
\end{equation*}
Here the last limit follows from $N_k\to\infty$ and the finiteness
of $\mathcal F_\Lambda$.
This proves Lemma \ref{lem:finite-denominator}.
\end{proof}

\begin{remark}\label{rmk:infinite-trace}
{\rm In the oscillation argument leading to Theorem \ref{thm:pointwise-main},
strong ergodicity and the assumption $\tau(1)=1$ are used only to ensure
that $\nu_x(\mathcal Q_q)=0$ whenever $\tau(x)=0$, as required in
Lemma \ref{lem:low-main}. 

When $\tau(1)=\infty$, ergodicity of the $\mathbb Z^d$-action $\gamma$
suffices to obtain bilaterally almost uniform convergence. Indeed, for every $x\in \mathcal{S}_\mathcal{M}$, the associated scalar
spectral measure $\nu_x$ on $\mathbb T^d$ is diffused, and hence
$\nu_x(\mathcal Q_q)=0$. The same oscillation argument therefore
applies and yields $M_\lambda^\gamma(x)\to0\quad \mathrm{b.a.u}$.}

\end{remark}

\noindent {\bf Acknowledgements.} 
This work was partially supported by the National Natural Science
Foundation of China (grant nos. {12601243,} 12325105,
W2441002, and W2611005).

\noindent\textbf{AI statement.}
The authors used 
GPT  only to assist with English-language
editing.

\end{document}